\documentclass[11pt,a4paper]{amsart}
\usepackage{amsmath,amssymb,mathrsfs,enumerate,booktabs,hyperref}
\usepackage[alphabetic]{amsrefs}

\input{preamble}
\mytheoremstyleenglish


\title{Extendability of automorphisms of K3 surfaces}

\author{Yuya Matsumoto}
\date{2021/01/07}
\address{\tusaddressfull}
\email{\gmail}
\email{\tusmail}
\thanks{This work was supported by JSPS Program for Advancing Strategic International Networks to Accelerate the Circulation of Talented Researchers,
	and by JSPS KAKENHI Grant Numbers 15H05738, 16K17560, and 20K14296.}


\begin{document}

\begin{abstract}
A K3 surface $X$ over a $p$-adic field $K$ is said to have good reduction if it admits a proper smooth model over the ring of integers of $K$.
Assuming this, we say that a subgroup $G$ of $\mathrm{Aut}(X)$ is extendable 
if $X$ admits a proper smooth model equipped with $G$-action (compatible with the action on $X$).
We show that $G$ is extendable if it is of finite order prime to $p$ and acts symplectically (that is, preserves the global $2$-form on $X$).
The proof relies on birational geometry of models of K3 surfaces,
and equivariant simultaneous resolutions of certain singularities.
We also give some examples of non-extendable actions.
\end{abstract}

\maketitle

\section{Introduction}

Throughout this article,
$K$ is a complete discrete valuation field of characteristic $0$, 
$\cO_K$ is its valuation ring, and $k$ is its residue field of characteristic $p \geq 0$ which we assume to be perfect.

Let $X$ be a K3 surface over $K$ with good reduction.
In this paper we consider relations between the automorphism groups of $X$ and of its proper smooth models over $\cO_K$.

If $X$ is an abelian variety,
then a proper smooth model of $X$ satisfies the N\'eron mapping property,
hence any automorphism of $X$ extend to that of the model.
To the contrary, 
a proper smooth model of a K3 surface does not in general satisfy the N\'eron mapping property,
due to the existence of flops,
and this makes automorphisms of $X$ not extendable in general to proper smooth models $\cX$ of $X$. 

Our main results are the following two theorems.
One gives a sufficient condition for an action to be extendable,
and the other gives examples that are not extendable.
Here we say that $G$ is \emph{extendable} if $X$ admits a proper smooth model equipped with a $G$-action extending that on $X$.
For precise definitions see Section \ref{sec:specialization}.

\begin{thm} \label{thm:sympfin}
Let $G \subset \Aut(X)$ be a symplectic finite subgroup of order prime to $p$.
Then $G$ is extendable.
\end{thm}
This fails without the assumptions, as the next theorem shows.

\begin{thm} \label{thm:non-extendable}
Let $p \geq 2$ be a prime.

(1) Let $G$ be either 
$\bZ/p\bZ$ (in which case we assume $p \leq 7$) or
$\bZ$. 
Then there exists a K3 surface $X$ defined and having good reduction over a finite extension $K$ of $\bQ_p$, 
equipped with a faithful symplectic action of $G$ that is not extendable.

(2) Let $G$ be either
$\bZ/p\bZ$ (in which case we assume $p \leq 19$),
$\bZ/l\bZ$ ($l$ a prime $\leq 11$ and $l \neq p$),
or $\bZ$.
Then the same conclusion holds, this time with a non-symplectic action.
\end{thm}

Here a group of automorphisms of a K3 surface is said to be \emph{symplectic}
if it acts on the $1$-dimensional space $H^0(X, \Omega^2_{X/K})$ trivially.
It is known that if a symplectic (resp.\ non-symplectic) automorphism of a K3 surface in characteristic $0$ has a finite prime order $l$
then $l \leq 7$ (resp.\ $l \leq 19$).
So Theorem \ref{thm:non-extendable} gives examples in most of the cases where Theorem \ref{thm:sympfin} does not apply.
For automorphisms of orders $13$, $17$, and $19$, see Proposition \ref{prop:131719}.

Let us now explain the strategy of the proof.
Using generalizations of results of Liedtke--Matsumoto \cite{Liedtke--Matsumoto} on birational geometry of models of K3 surfaces to equivariant settings (Section \ref{sec:G-flops}), 
we reduce Theorem \ref{thm:sympfin} and a part of Theorem \ref{thm:non-extendable}
to the following local result on simultaneous equivariant resolution,
which may be of independent interest.
\begin{thm} \label{thm:localresolution}
Let $(B,\fm)$ be a flat local $\cO_K$-algebra
of relative dimension $2$ obtained as 
the localization of a finite type $\cO_K$-algebra at a maximal ideal,
with $B/\fm \cong k$, 
$B \otimes K$ smooth, and $B \otimes k$ an RDP (rational double point).
Let $G$ be a nontrivial finite group of order prime to $p$ acting on $B$ over $\cO_K$ faithfully.
Then $B$ admits a simultaneous $G$-equivariant resolution in the category of algebraic spaces after replacing $K$ by a finite extension
if and only if the $G$-action is symplectic (in the sense of Definition \ref{defn:symplectic}(\ref{defn:symplectic:3})).
\end{thm}

Here a simultaneous resolution is a proper morphism $\cX \to \Spec B$
which is an isomorphism on the generic fiber and the minimal resolution on the special fiber.
We prove Theorem \ref{thm:localresolution} in Section \ref{sec:localresolution}
by giving a classification of symplectic actions (Proposition \ref{prop:local action}) and case-by-case explicit simultaneous resolutions (Proposition \ref{prop:local resolution}).

Currently we do not have any explanation why symplecticness arises as a key condition.
It may be related to the fact that the RDPs in characteristic $0$ are precisely 
the quotient singularities by ``symplectic'' group actions (cf.\ proof of Proposition \ref{prop:symplecticRDP}).

To prove other cases of Theorem \ref{thm:non-extendable} 
we define in Section \ref{sec:specialization} the \emph{specialization map} $\spen \colon \Aut(X) \to \Aut(\cX_0)$ 
($\cX_0$ is the special fiber of $\cX$) and show that
if $g$ is extendable 
then the characteristic polynomials of $g^*$ and $\spen(g)^*$ on $\Het^2$ should coincide (Proposition \ref{prop:criterion for non-ext}).
In Section \ref{sec:proof} we give examples in which these polynomials differ.

As a side trip,
we study this specialization map $\spen \colon \Aut(X) \to \Aut(\cX_0)$.
As will be seen in Section \ref{sec:kersp}, $\Ker (\spen)$ may have nontrivial members, both of finite and infinite orders.
We show that if a finite order automorphism is in $\Ker (\spen)$ then its order is a power of the residue characteristic $p$ (Proposition \ref{prop:finiteorder}).
Such automorphisms are related to actions of infinitesimal group schemes such as $\mu_p$ and $\alpha_p$, which we will investigate in future papers.
In Section \ref{sec:3} 
we also give an example where the characteristic polynomial of the action of $\spen(g)^*$ on $\Het^2$
is irreducible (which never happens on $\Het^2$ of a K3 surface in characteristic $0$).

\subsection*{Acknowledgments}
I thank Keiji Oguiso for the interesting question on extendability of automorphism groups, from which this work arose.
I thank H\'el\`ene Esnault, Christian Liedtke, Yuji Odaka, and Nicholas Shepherd-Barron for their helpful comments.
I appreciate the kind hospitality of Institut de Math\'ematiques de Jussieu-Paris Rive Gauche where a large part of this work was done.

\section{Specialization of automorphisms of K3 surfaces} \label{sec:specialization}

\begin{defn}
Let $X$ be a proper surface over $K$.

(1) A \emph{model} of $X$ over $\cO_K$ is a proper flat algebraic space $\cX$ over $\cO_K$ 
equipped with an isomorphism $\cX \times_{\cO_K} K \isomto X$.
A \emph{projective smooth model} is a model that is projective and smooth over $\cO_K$, and so on.
Note that a model may not be a scheme, 
but a projective model is always a scheme.

(2) We say that $X$ has \emph{good reduction} if $X$ admits a proper smooth model. 
We say that $X$ has \emph{potential good reduction} if $X_{K'}$ has good reduction for some finite extension $K'/K$.

(3) Let $G$ be a subgroup of $\Aut(X)$.
A \emph{$G$-model}
is a model of $X$ equipped with a $G$-action compatible with that of $X$.
If $G$ is generated by a single element $g$, we also call it a \emph{$g$-model}.

(4) We say that $G \subset \Aut(X)$ (resp.\ $g \in \Aut(X)$) is \emph{extendable} 
if, after replacing $K$ by a finite extension,
$X$ admits a proper smooth $G$- (resp.\ $g$-) model. 
\end{defn}

We also introduce a related notion of \emph{specialization} of automorphisms.

\begin{prop} \label{prop:def of sp}
Let $X$ be a K3 surface over $K$ having good reduction. 

(1) For any proper smooth model $\cX$ of $X$, 
an automorphism $g$ of $X$ extend to a unique birational (rational) self-map of $\cX$
and its locus of indeterminacy is a closed subspace of codimension at least $2$.
The induced birational self-map on the special fiber $\cX_0$ is in fact an automorphism,
which we write $\spen(g)$ and call the specialization of $g$.

(2) Both the special fiber $\cX_0$ and the specialization morphism $\spen \colon \Aut(X) \to \Aut(\cX_0)$
are independent of the choice of the model $\cX$. 
This map $\spen$ is a group homomorphism.
\end{prop}

\begin{proof}
(1)
Take $g \in \Aut(X)$.
Let $g^* \cX$ be the normalization of $\cX$ in the pullback $g \colon X \to X$.
Then $g^* \cX$ is another proper smooth model and it is connected to $\cX$ by a finite number of flopping contractions (\cite{Liedtke--Matsumoto}*{Proposition 4.7}).
It follows that $g$ induces a birational self-map on $\cX$ with indeterminacy of codimension at least $2$.

The restriction of $g$ to the special fiber $\cX_0$ is a birational self-map, and in fact an isomorphism since $\cX_0$ is minimal.

(2)
This again follows from the fact that two proper smooth models of $X$ are isomorphic outside subspaces of codimension $\geq 2$.
\end{proof}

\begin{prop} \label{prop:criterion for non-ext}
Let $X$ be a K3 surface over $K$ having good reduction.
Let $g \in \Aut(X)$ 
and let $\spen(g) \in \Aut(\cX_0)$ be its specialization.
Assume that the characteristic polynomials of 
$g^*$ and $\spen(g)^*$ on $\Het^2(X_{\overline K}, \bQ_l)$ and $\Het^2((\cX_0)_{\overline k}, \bQ_l)$ do not coincide.
Then $g$ is not extendable. 
\end{prop}

\begin{proof}
The proper smooth base change theorem 
induces, for each proper smooth model $\cX$,
an isomorphism between $\Het^2(X_{\overline K}, \bQ_l)$ and $\Het^2((\cX_0)_{\overline k}, \bQ_l)$.
In general this isomorphism depends on the choice of the model.
If $\cX$ admits a $g$-action then this isomorphism is $g$-equivariant,
and then the characteristic polynomials of $(g \rvert X)^*$ and $(g \rvert \cX_0)^*$ coincide.
(We have $g \rvert \cX_0 = \spen(g \rvert X)$ by definition.)
\end{proof}

\begin{rem}
This proposition cannot give a counterexample to Theorem \ref{thm:sympfin} since,
under the assumption of the theorem, the characteristic polynomials always coincide
by Lemma \ref{lem:sympfin} and Proposition \ref{prop:finiteorder}.

We do not know whether the converse holds, i.e.\ whether the coincidence of characteristic polynomials implies extendability.
\end{rem}

\begin{cor} \label{cor:criterion for non-ext}\ 
\begin{enumerate}
\item \label{item:specialization is trivial}
Let $X$ and $g$ as in Proposition \ref{prop:criterion for non-ext}.
Assume $g \neq \id$ and $\spen(g) = \id$.
Then $g$ does not extend to any proper smooth model of $X$.

\item \label{item:specialization is irreducible}
Let $X_0$ be a K3 surface over $k$ and let $g_0 \in \Aut(X_0)$.
Assume that the characteristic polynomial of $g_0^*$ on $\Het^2$ is irreducible over $\bZ$. 
Then $g_0$ is not the restriction of any automorphism of 
any proper smooth model $\cX$ of any K3 surface $X$ over any $K$ (of characteristic $0$).
\end{enumerate}
\end{cor}

\begin{proof}
(1)
By the Torelli theorem, nontrivial $g$ acts nontrivially on $\Het^2$. 
Hence the assertion follows from Proposition \ref{prop:criterion for non-ext}.

(2)
In characteristic $0$ the characteristic polynomial cannot be irreducible since
both $\NS(X) \otimes \bQ_l \subset \Het^2$ and its orthogonal complement $T$ are nontrivial subspaces.
\end{proof}

\begin{rem}
If the condition of (2) is satisfied
then $X_0$ is supersingular and 
the characteristic polynomial is a Salem polynomial (Lemma \ref{lem:Salem}).
We will see in Section \ref{sec:3}
that such $g_0$ still may be the specialization of an automorphism in characteristic $0$. 
\end{rem}

In practice it is easier to compute the specialization map if we use more general models than the smooth ones. 
\begin{defn} \label{def:RDP}
(1) An \emph{RDP surface} over a field $F$ is a surface $X$ such that 
$X_{\overline F}$ has only RDP (rational double point) singularities.

(2) An \emph{RDP K3 surface} over a field is a proper RDP surface
whose minimal resolution is a K3 surface. 
(In particular, a smooth K3 surface is an RDP K3 surface by definition.)

(3) A \emph{proper RDP model} of an RDP K3 surface
is a proper model whose special fiber is an RDP surface.
(The special fiber is then an RDP K3 surface. This follows from the next lemma and the classification of degeneration of K3 surfaces.)

(4) A \emph{simultaneous resolution}
of a proper RDP model $\cX$ of an RDP K3 surface 
is a proper morphism $f \colon \cY \to \cX$ from an algebraic space
that is the minimal resolution on each fiber.
\end{defn}

Note that for an RDP K3 surface $X$ there is a canonical injection $\Aut(X) \to \Aut(\tilde X)$,
where $\tilde X$ is the minimal resolution.

\begin{lem} \label{lem:resolve rdp}
If an RDP K3 surface $X$ admits a proper RDP model, 
then the minimal resolution $\tilde X$ of $X$ has potential good reduction.

More precisely,
if $\cX$ is a proper RDP model of $X$ over $\cO_K$,
then after extending $K$ there exists a simultaneous resolution $\cY \to \cX$ 
and then $\cY$ is a proper smooth model of $\tilde X$.
\end{lem}

\begin{proof}
By extending $K$, we may assume that all singular points of $X$ are $K$-rational.
If $X$ is not smooth, take an RDP $x \in X$,
and let $\pi \colon \cX' \to \cX$ be the blow-up at the Zariski closure $Z$ of $\{ x \}$.
Then $Z \cap \cX_0$ consists of an RDP $ x_0 $ and
the restriction of $\pi$ on the generic (resp.\ special fiber) is the blow-up at $x$ (resp.\ $x_0$).
Hence $\cX'$ is again a proper RDP model of an RDP K3 surface.
Repeating this, we may assume the generic fiber $X$ is smooth.

If the generic fiber is smooth,
then \cite{Artin:brieskorn}*{Theorem 2} gives a (non-canonical) simultaneous resolution.
\end{proof}

\begin{prop} \label{prop:compute sp}
Let $\cX_1,\cX_2$ be proper RDP models of RDP K3 surfaces $X_1,X_2$
and $Z_i \subset \cX_i$ closed subspaces that do not contain the special fiber $(\cX_i)_0$.
Let $g \colon \cX_1 \setminus Z_1 \to \cX_2 \setminus Z_2$
be a birational morphism.
Then the specialization of the induced automorphism $\tilde X_1 \isomto \tilde X_2$
is the automorphism induced by 
$g \restrictedto{(\cX_1 \setminus Z_1)_0} \colon (\cX_1 \setminus Z_1)_0 \to (\cX_2 \setminus Z_2)_0$.
\end{prop}

\begin{proof}
Proper RDP models $\cX_i$ have simultaneous resolutions $\cY_i \to \cX_i$.
By adding the exceptional loci of these morphisms into $Z_i$,
we may assume that $\cX_i$ themselves are smooth.
Since $\cX_1$ and $\cX_2$ are isomorphic outside closed subspaces of codimension $\geq 2$ (\cite{Liedtke--Matsumoto}*{Proposition 4.7}),
we may assume $\cX_1 = \cX_2$.
Then the birational self-map of $\cX_1$ in Proposition \ref{prop:def of sp} is the one induced by $g$.
\end{proof}

We also need the relation between $\Omega^2$ of the fibers of proper RDP models.

\begin{lem} \label{lem:Omega2-CI}
Let $(C,\fn)$ an $m$-dimensional local ring of the (complete intersection) form $C = k[x_1, \ldots, x_{n+m}]_0 / (F_1, \ldots, F_n)$
where $_0$ is the localization at the origin,
and assume $U = \Spec C \setminus \{ \fn \}$ is smooth.
Then there exists a unique element $\omega \in \Gamma(U, \Omega^m_{C/k})$ such that for any $\sigma \in \fS_{n+m}$ the equality 
$\sign(\sigma) \det((F_j)_{x_{\sigma(i)}})_{i,j=1}^{n} \omega = dx_{\sigma(n+1)} \wedge \cdots \wedge dx_{\sigma(n+m)}$ holds,
and such $\omega$ generates $\Omega^m_{C/k} \restrictedto{U}$.

The same holds if we replace $k[\ldots]_0$ with
its Henselization $k\polyHens{\ldots}$ or completion $k[[\ldots]]$.
\end{lem}
Here $F_{x_i}$ is defined by the equality
$dF = \sum_i F_{x_i} dx_i$ in $\Omega^1_{k[\ldots]_0/k}$ (or in ...).
This coincides with the termwise partial differentiation of formal power series.

\begin{proof}
Straightforward.
Note that at every point on $U$, we have $\det((F_j)_{x_{\sigma(i)}}) \neq 0$ for some $\sigma \in \fS_{n+m}$.
\end{proof}

\begin{lem} \label{lem:Omega2-RDP}
Let $(C, \fn)$ be a $2$-dimensional local ring over a field $k$
and assume it is an RDP.
Define $U$ as above.

(1) $\Omega^2_{C/k} \restrictedto{U}$ is trivial,
and hence $H^0(U, \Omega^2_{C/k}) \cong H^0(U, \cO) = C$.

(2)
Let $\pi \colon X \to \Spec C$ be the minimal resolution.
Then $H^0(X, \Omega^2_{X/k}) \to H^0(\pi^{-1}(U), \Omega^2_{X/k}) \isomto H^0(U, \Omega^2_{C/k})$ 
is an isomorphism.
\end{lem}
\begin{proof}
It suffices to show the assertion after taking \'etale local base change $C \to C'$;
Hence we may assume $C$ is of the form $C = k \polyHens{x_1,x_2,x_3} / (F)$, $F \in (x_1,x_2,x_3)^2$, $F \not \in (x_1,x_2,x_3)^3$
(\cite{Lipman:rationalsingularities}*{Lemma 23.4}).

(1) Indeed, $\Omega^2_{C/k} \restrictedto{U}$ is generated by $\omega$ defined above.

(2)
Let $C_1 = k \polyHens{x_1, x_2/x_1, x_3/x_1} / (F/x_1^2)$
be the first affine piece of $\Bl_{(x_1,x_2,x_3)} C$,
and define $C_2,C_3$ similarly.
Define $\omega$ and $\omega_i$ as in the previous lemma.
Then we have $\omega_i = \omega$.
If all $C_i$ are smooth (hence $X = \bigcup \Spec C_i$) then we have 
$H^0(X, \Omega^2_{X/k}) = C_1 \omega_1 \cap C_2 \omega_2 \cap C_3 \omega_3 = C \omega$.
General case follows inductively from this.
\end{proof}

\begin{lem} \label{lem:Omega2-RDPmodel}
Let $\cX$ be a proper RDP scheme model over $\cO_K$ of an RDP K3 surface $X$
and $\Sigma \subset \cX$ the closed subset of RDPs.
Then $H^0(\cX \setminus \Sigma, \Omega^2_{\cX / \cO_K})$ is free $\cO_K$-module of rank $1$, with generator say $\omega$, 
and $H^0(\cX_0 \setminus \Sigma_0, \Omega^2_{\cX_0 / k})$ and $H^0(\tilde \cX_0, \Omega^2_{\tilde \cX_0 / k})$ is generated by (the restriction of) $\omega$,
where $\tilde \cX_0$ is the minimal resolution.
If $\cX$ admits an automorphism $g$, 
then this is compatible with the action of the automorphisms $g \restrictedto{X}$ and $g \restrictedto{\cX_0} = \spen(g \restrictedto{X})$.
\end{lem}
\begin{proof}
We have $\dim H^0(X \setminus \Sigma_K, \Omega^2_{X / K}) = \dim H^0(\cX_0 \setminus \Sigma_0, \Omega^2_{\cX_0 / k}) = 1$ from the previous lemma.
The former assertion follows from this and upper semi-continuity and the previous lemma.
The latter is clear.
\end{proof}

We recall a result on the trace of finite order symplectic automorphisms.
For a positive integer $n \leq 8$, 
define $\varepsilon(n)$ so that 
\[
\frac{1}{\varepsilon(n)} = \frac{n}{24} \prod_{q:\text{prime},q \divides n} \Bigl( 1+\frac{1}{q} \Bigr).
\]
We have $\varepsilon(n) = 24,8,6,4,4,2,3,2$ for $n = 1,2,3,4,5,6,7,8$ respectively.
\begin{lem} \label{lem:sympfin}
Let $X$ be a K3 surface over a field $F$ of characteristic $p \geq 0$ 
and $g \in \Aut(X)$ a nontrivial symplectic automorphism of finite order prime to $p$.
Then $\ord(g) \leq 8$,
the fixed points of $g$ are isolated,
and $\card{\Fix(g)} = \varepsilon(\ord(g))$.
Moreover the trace of $g^*$ on $\Het^2(X_{\overline F}, \bQ_l)$ (and on $H^2(X, \bQ)$ if $F = \bC$) depends only on $\ord(g)$
and is equal to $\varepsilon(\ord(g)) - 2$.
(In other words, the characteristic polynomial of $g^*$ on $\Het^2$ depends only on $\ord(g)$.)
\end{lem}
The equality $\trace(g) = \varepsilon(\ord(g)) - 2$ holds also if $\ord(g) = 1$.

\begin{proof}
Characteristic $0$: 
\cite{Nikulin:auto}*{Section 5 and Theorem 4.7} proves everything except the value of the trace.
\cite{Mukai:automorphismsK3}*{Propositions 1.2, 3.6, 4.1} proves everything.

Characteristic $p > 0$:
\cite{Dolgachev--Keum:auto}*{Theorem 3.3 and Proposition 4.1}.
\end{proof}

\begin{cor} \label{cor:Picardnumber}
Let $X$ is a K3 surface over a field $F$ of characteristic $0$ 
and $G \subset \Aut(X)$ a nontrivial finite group of symplectic automorphisms. 
Define $\mu(G) = \card{G}^{-1} \sum_{g \in G} \varepsilon(\ord(g))$.
Then the (geometric) Picard number of $X$ is at least $25 - \mu(G)$.
\end{cor}
\begin{proof}
We may assume $F = \bC$.
Let $V$ be the $G$-representation $H^2(X, \bQ)$.
By the previous lemma $\trace (V, g) = \varepsilon(\ord(g)) - 2$.
Let $\{ \rho \}$ be the set of irreducible representations of $G$
and write $V = \sum a_{\rho} \rho$, $a_\rho \in \bZ_{\geq 0}$.
Then we have $a_1 = (1 \cdot V) = \card{G}^{-1} \sum_{g \in G} \trace (V, g) = \card{G}^{-1} \sum_{g \in G} (\varepsilon(\ord(g)) - 2)  = \mu(G) - 2$
(here $1$ denotes the trivial representation).
Since $G$ acts trivially on the transcendental lattice $T(X)$
and $G$ has nontrivial invariant subspace in $\NS(X)$,
we have $\rank (T(X)) \leq a_1 - 1$.
\end{proof}

\section{Local equivariant simultaneous resolutions} \label{sec:localresolution}

In this section we prove Theorem \ref{thm:localresolution}.
In the symplectic case, 
we first classify possible actions and give explicit equations (Proposition \ref{prop:local action}) by using a versal equivariant deformation (Theorem \ref{thm:versal G-deformation}),
and then give explicit equivariant simultaneous resolutions (Proposition \ref{prop:local resolution}).

We often apply the following approximation lemma to 
the Henselization $A = R \polyHens{x_1, \ldots, x_n}$ of $R[x_1, \ldots, x_n]$ at the origin,
where $R = k$ or $R = \cO_K$, and $I = (x_1, \ldots, x_n)$.
\begin{lem}[\cite{Artin:approximation}*{Theorem 1.10}] \label{lem:approximation}
Let $R$ be a field or an excellent discrete valuation ring.
Let $A$ be the Henselization of a finite type $R$-algebra at a prime ideal and $I \subset A$ a proper ideal (not necessarily the maximal ideal).
Given a system $f_j(Y) = 0$ ($Y = (Y_1, \ldots, Y_N)$) of polynomial equations with coefficients in $A$,
a solution $\overline y$ in the $I$-adic completion $\hat A$ of $A$, and an integer $c$,
there exists a solution $y$ in $A$ with $\overline y_i \equiv y_i \pmod{I^c}$.
\end{lem}

We begin with the definition of symplecticness of automorphism of local rings
(which will be seen later to be compatible with that of K3 surfaces).

\begin{defn} \label{defn:symplectic}\ 
\begin{enumerate}
\item \label{defn:symplectic:2}
Let $(C,\fn)$ be a $2$-dimensional normal local ring over a field $k$ with isolated Gorenstein singularity (e.g.\ RDP) with $C/\fn \cong k$.
Let $U = \Spec C \setminus \{ \fn \}$.
Then $\Omega^2_{C/k} \restrictedto{U}$ is trivial,
and hence $H^0(U, \Omega^2_{C/k}) \cong H^0(U, \cO) = C$.
We say that an automorphism or a group of automorphisms of $C$ over $k$ is \emph{symplectic}
if it acts on the $1$-dimensional $k$-vector space $H^0(U, \Omega^2_{C/k}) \otimes_C C/\fn$ trivially.
\item \label{defn:symplectic:3}
Let $B$ be as in Theorem \ref{thm:localresolution}.
We say that an automorphism of $B$ over $\cO_K$ is \emph{symplectic} if the induced automorphism of $B \otimes k$ is so.
\end{enumerate}
\end{defn}

In some cases we can compute $\Omega^2_{C/k} \restrictedto{U}$ and the action on it explicitly:
If $C$ is as in Lemma \ref{lem:Omega2-CI}, 
and $g$ is an automorphism of $C$ with $g(x_i) = a_i x_i$ and $g(F_j) = e_j F_j$ for some $a_i,e_j \in k^*$,
then $g(\omega) = (\prod a_i / \prod e_j) \omega$, 
and in particular $g$ is symplectic if and only if $\prod a_i = \prod e_j$.

\begin{lem} \label{lem:fixisisolated}
Let $C,U$ be as in Lemma \ref{lem:Omega2-RDP}. 
$X \to \Spec C$ the minimal resolution,
and let $g \in \Aut(C)$ a nontrivial symplectic automorphism of finite order prime to $p = \charac k$.
Then $g$ acts on $X$ and $\Fix(g) \subset X$ is $0$-dimensional (if nonempty).
\end{lem}
\begin{proof}
Let $x \in X$ be a fixed closed point.
Since $g$ is of finite order prime to $p$, 
the action of $g$ on $T^*_{X,x}$ is semisimple (diagonalizable).
By Lemma \ref{lem:Omega2-RDP},
this action has determinant $1$
(since $\Omega^2_{X,x} \cong \det T^*_{X,x}$)
and hence its eigenvalues are of the form $\lambda,\lambda^{-1}$.
Since $g \neq 1$ we have $\lambda, \lambda^{-1} \neq 1$.
This implies $x$ is isolated in $\Fix(g)$. 
\end{proof}

\begin{prop} \label{prop:tame}
	Assume that $C$ is moreover an RDP,
	and that a finite group $G$ of order not divisible by $p = \charac k$
	acts on $C$ symplectically.
	Then the invariant ring $C^G$ is again an RDP.
	
	Let $X = \Spec C$ and let $\tilde{X} \to X$ be the minimal resolution.
	Then $\tilde{X}/G \to X/G$ is crepant.
\end{prop}

\begin{proof}
	Let $\omega$ be a generator of the rank $1$ free $C$-module $H^0(\Spec C \setminus \set{\fm}, \Omega^2_{C/k})$.
	The action of $G$ on $X = \Spec C$ induces an action on the minimal resolution $\tilde{X}$
	and $\omega$ extends to a regular non-vanishing $2$-form on $\tilde{X}$.
	At each closed point $z \in \tilde{X}$ the stabilizer $G_z \subset G$ acts on $T_z \tilde{X}$ via $\SL_2(k)$ since $G$ preserves $\omega$.
	Hence the quotient $\tilde{X}/G$ has only RDPs as singularities. 
	Since $\omega$ is preserved by $G$ it induces a regular non-vanishing $2$-form on $(\tilde{X}/G)^\sm$, and
	since RDPs are canonical singularities it extends to a regular non-vanishing $2$-form on the resolution $\widetilde{\tilde{X}/G}$ of $\tilde{X}/G$.
	Thus $C^G$ is a canonical singularity, that is, either a smooth point or an RDP.	
	Since $G \neq \set{1}$, $C^G$ cannot be smooth.
\end{proof}

\begin{lem} \label{lem:symplectic}\ 
\begin{enumerate}
\item 
Let $X_0$ be an RDP K3 surface over a field $k$, 
$x \in X_0(k)$ an RDP or a smooth point,
and $G \subset \Aut(X_0)$ a subgroup fixing $x$.
Let $\tilde X_0$ be the minimal resolution of $X_0$ (then we have natural injection $\Aut(X_0) \to \Aut(\tilde X_0)$).
Then $G$ is symplectic as a subgroup of $\Aut(\tilde X_0)$
if and only if it is symplectic as a subgroup of $\Aut(\cO_{X_0,x})$ in the sense of Definition \ref{defn:symplectic}(\ref{defn:symplectic:2}).
\item \label{lem:symplectic:3}
Let $\cO_K$ be as above.
Let $\cX$ be a proper RDP model of an RDP K3 surface $X$ over $K$,
$x \in \cX(k)$ an RDP or a smooth point of $\cX_0$,
and $G \subset \Aut(\cX)$ a subgroup fixing $x$.
Assume that $G$ is finite and of order prime to $p = \charac k$.
Then $G$ is symplectic as a subgroup of $\Aut(\tilde X)$
if and only if it is symplectic as a subgroup of $\Aut(\cO_{\cX_0,x})$ in the sense of Definition \ref{defn:symplectic}(\ref{defn:symplectic:3}).
\end{enumerate}
\end{lem}
\begin{proof}
(1) 
Let $C = \cO_{X_0,x}$ and define $\fn$ and $U$ as above.
Let $\omega$ be a nonzero element (hence a generator) of $H^0(\tilde X_0, \Omega^2)$.
Then $\omega$ restricts to a generator of $H^0(U, \Omega^2_{C/k}) \otimes_C C/\fn$,
hence the action of $G$ on the two spaces coincide.

(2) 
Take a generator $\omega$ of $H^0(\cX \setminus \Sigma, \Omega^2)$ (Lemma \ref{lem:Omega2-RDPmodel}), where $\Sigma \subset \cX_0$ is the set of RDPs.
The action of $G \subset \Aut(\cX)$ on $\omega \restrictedto{\tilde X}$ factors through $\mu_N(K)$ for some $N$ prime to $p$.
On the other hand $\omega \restrictedto{\cX_0}$ restricts to a generator of $H^0(U, \Omega^2_{C/k}) \otimes_C C/\fn$,
where $C = \Spec \cO_{\cX_0,x}$.
The action of $G$ on the two spaces are compatible under the reduction map $\mu_N(K) \to \mu_N(k)$.
This map is injective since $N$ is prime to $p$.
\end{proof}

First we consider the symplectic case of Theorem \ref{thm:localresolution}. 
We use the following classification of symplectic actions (Proposition \ref{prop:local action}) and case-by-case explicit simultaneous resolutions (Proposition \ref{prop:local resolution}).
We say that two pairs $(G_i, B_i)$ ($i = 1,2$)
of a finite group $G_i$ and a local $\cO_K$-algebra $B_i$ equipped with a $G_i$-action
are \emph{\'etale-locally isomorphic}
if there exists a pair $(G_3, B_3)$,
group isomorphisms $G_i \isomto G_3$, 
and equivariant \'etale local morphisms $B_i \to B_3$ of local $\cO_K$-algebras.

\begin{prop} \label{prop:local action}
Let $B$ and $G$ be as in Theorem \ref{thm:localresolution},
and assume $G$ is symplectic.
Then 
$(G, \Sing(B_0))$ is one of the pairs listed below.
Moreover,
except for the cases where $(G, \Sing(B_0)) = (\Tet, A_1), (\Oct, A_1), (\Ico, A_1)$,
the pair $(G,B)$ is \'etale-locally isomorphic to the normal form $(G',B')$, 
$B' = \cO_K \polyHens{x,y,z} / (F)$ 
with $F$ and $G'$-action described below,
after replacing $K$ by a finite extension.

In each case below, 
$q_l$ are some elements of the maximal ideal $\fp$ of $\cO_K$.

\begin{itemize}
\item 
[$(C_2, E_6)$]
$F$ is one of the following, and the nontrivial element of $G' = C_2$ acts by $(x,y,z) \mapsto (-x, y, -z)$.
\begin{itemize}
\item[$(E_6)$] ($p \neq 3$): $F = x^2 + y^3 + z^4 + q_{00} + q_{10} y + q_{02} z^2 + q_{12} y z^2$.
\item[$(E_6^0)$]  ($p = 3$): $F = x^2 + y^3 + z^4 + q_{00} + q_{10} y + q_{20} y^2 + q_{02} z^2 + q_{12} y z^2 + q_{22} y^2 z^2$.
\item[$(E_6^1)$]  ($p = 3$): $F = x^2 + y^3 + y^2 z^2 + z^4 + q_{00} + q_{10} y + q_{20} y^2 + q_{02} z^2$.
\end{itemize}

\item 
[$(C_2, D_m)$] 
$m \geq 4$,
$F = x^2 + yz^2 + y^{m-1} + \sum_{l = 0}^{m-2} q_l y^l$,
and the nontrivial element of $G' = C_2$ acts by $(x,y,z) \mapsto (-x, y, -z)$.

\item 
[$(\fS_3, D_4)$, $(\fA_3, D_4)$]
$F$ is one of the following, 
$G'$ is either $\fS_3$ or $\fA_3$,
and $G' \subset \fS_3$ acts by 
$(123)(x, y, z) = (x, \zeta_3 y, \zeta_3^{-1} z)$,
$(12)(x, y, z) = (-x, z, y)$.
\begin{itemize}
\item[$(D_4)$] ($p \neq 2$): $F = x^2 + y^3 + z^3 + q_{000} + q_{011} yz$.
\item[$(D_4^0)$]  ($p = 2$): $F = x^2 + y^3 + z^3 + q_{000} + q_{100} x + q_{011} yz + q_{111} xyz$.
\item[$(D_4^1)$]  ($p = 2$): $F = x^2 + y^3 + z^3 + xyz + q_{000} + q_{100} x $.
\end{itemize}

We also have an alternative form: 
$B' = \Spec \cO_K \polyHens{x,y_1,y_2,y_3} / (F_1,F_2)$,
$F_1 = y_1 y_2 y_3 + Q(x)$, $F_2 = y_1 + y_2 + y_3 - R(x)$,
where $Q(x),R(x) \in \cO_K[x]$ are polynomials
of the following form with $q'_l,r'_l \in \fp$,
and $G' \subset \fS_3$ acts by $\rho(x) = \sign(\rho)x, \rho(y_i) = y_{\rho(i)}$.
\begin{itemize}
\item[$(D_4)$] ($p \neq 2$): $Q(x) = x^2 + q'_0$, $R(x) = r'_0$.
\item[$(D_4^0)$]  ($p = 2$): $Q(x) = x^2 + R(x)^3 + \sum_{l=0}^{1} q'_l x^l$, $R(x) = \sum_{l=0}^{1} r'_l x^l$.
\item[$(D_4^1)$]  ($p = 2$): $Q(x) = x^2 + R(x)^3 + \sum_{l=0}^{1} q'_l x^l$, $R(x) = x$.
\end{itemize}

\item 
[$(\Dih_n, A_{m-1})$] 
$m \geq 2$ even, $n \geq 1$,
$F = xy + z^{m} + \sum_{l = 0}^{m-1} q_{l} z^{l}$, 
$q_l = 0$ if $l$ odd, and
$G' = \Dih_n$
acts by $\sigma(x,y,z) = (\zeta_n x, \zeta_n^{-1} y, z)$ and $\tau(x,y,z) = (y, x, -z)$.
\item 
[$(\Dic_{n}, A_{m-1})$] 
$m \geq 3$ odd, $n \geq 2$ even, 
$F = xy + z^{m} + \sum_{l = 0}^{m-1} q_{l} z^{l}$,
$q_l = 0$ if $l$ even, and
$G' = \Dic_{n}$
acts by $\sigma(x,y,z) = (\zeta_{n} x, \zeta_{n}^{-1} y, z)$ and $\tau(x,y,z) = (y, -x, -z)$.
\item 
[$(C_n, A_{m-1})$] 
$m \geq 2$, $n \geq 2$,
$F = xy + z^{m} + \sum_{l = 0}^{m-1} q_{l} z^{l}$,
$q_{m-1} = 0$ if $p$ does not divide $m$, and 
$G' = C_n$ is the cyclic group of order $n$ with generator $\sigma$
acting by $\sigma(x,y,z) = (\zeta_n x, \zeta_n^{-1} y, z)$.
\item
[$(G, A_1)$]
$G$ is $\Tet$, $\Oct$, or $\Ico$.
\end{itemize}
Here $\zeta_n$ is a primitive $n$-th root of unity;
$C_n$ is the cyclic group of order $n$;
\begin{align*}
\Dih_n &= \spanned{\sigma,\tau \mid \sigma^n = \tau^2 = \tau\sigma\tau^{-1}\sigma = 1}, \\
\Dic_n &= \spanned{\sigma,\tau \mid \sigma^n = \sigma^{n/2} \tau^2 = \tau\sigma\tau^{-1}\sigma = 1}
\end{align*}
are respectively the dihedral and dicyclic groups (of order $2n$),
where $n$ is assumed to be even for $\Dic_n$;
and $\Tet$, $\Oct$, and $\Ico$ are respectively 
the tetrahedral, octahedral, and icosahedral groups (of order $12$, $24$, $60$).
\end{prop}

\begin{rem} \label{rem:localresolution} 
$E_6^0,E_6^1$ (in $p = 3$) and $D_4^0,D_4^1$ (in $p = 2$)
are analytically non-isomorphic RDPs having the same Dynkin diagrams.
See \cite{Artin:RDP} for the classification and notation.

We do not give a normal form of $B'$ in the cases $(G, A_1)$ ($G = \Tet, \Oct, \Ico$)
because our method using Theorem \ref{thm:versal G-deformation} fails for these groups (see Remark \ref{rem:finite generation of T}) and 
our proof of Proposition \ref{prop:local resolution} does not need one.

It is likely that, 
except for the case $(G, A_1)$ ($G = \Tet, \Oct, \Ico$),
the number of parameters $q_l$ in each case (excluding those indicated to be $0$)
coincide with the relative dimension of the deformation space of the singularity equipped with the group action,
cf.\ Theorem \ref{thm:versal G-deformation}.

Shepherd-Barron has recently announced \cite{Shepherd-Barron:Weyl group covers} that 
the set of simultaneous resolution of a deformation of an RDP (not equipped with a group action)
is a torsor of the Weyl group and in particular they have the same cardinality 
(this was known in complex case by Brieskorn \cite{Brieskorn:auflosung},\cite{Brieskorn:singularelements}).
Using this, we might be able to prove this proposition
by computing the $G$-action on this set and finding a fixed element.

It is likely that, 
under the assumption of good reduction (i.e.\ existence of simultaneous resolution that is not necessarily $G$-equivariant), 
there exists a simultaneous $G$-equivariant resolution without extending $K$.
We do not pursue this.
\end{rem}

\begin{prop} \label{prop:symplecticRDP}
Let $k$ be a perfect field of characteristic $p \geq 0$.
Let $C$ be a local $k$-algebra of relative dimension $2$ obtained as 
the localization of a finite type $k$-algebra at a maximal ideal, with an RDP singularity.
Let $G$ be a nontrivial finite group of order prime to $p$ acting on $C$ symplectically and faithfully.

Then 
$(G, \Sing(C))$ is one of the in the list of Proposition \ref{prop:local action}.
Moreover,
except for the cases where $(G, \Sing(C)) = (\Tet, A_1), (\Oct, A_1), (\Ico, A_1)$,
the pair $(G,C)$ is \'etale-locally isomorphic to the normal form  $(G',B' \otimes k)$ 
(so all of $q_l,q'_l,r'_l$ are $0$)
for one of $(G',B')$ in the list of Proposition \ref{prop:local action},
after replacing $k$ by a finite extension.
\end{prop}

\begin{proof}[Proof of Proposition \ref{prop:symplecticRDP}]
The (\'etale) fundamental group of a Henselian RDP $\Spec C$ is well-known in characteristic $0$,
and determined by Artin \cite{Artin:RDP}*{Sections 4--5} in characteristic $> 0$.
Here the fundamental group means 
$\pi_1(\Spec C \setminus \set{\fm})$ and is abbreviated as $\pi_1(C)$.
We summarize the result in Table \ref{table:pi1}.
Here, $p^e$ is read to be $1$ if $p = \charac k$ is zero,
and in any characteristic $A_0$ is read to be smooth.
For $D_N^r$ in characteristic $2$ with $2 \divides N$ and $4r > N$,
$2^e$ is the largest power of $2$ dividing $4r-N$, and  $(4r-N)'$ is the remaining factor of $4r-N$, 
i.e.\ $4r - N = 2^e (4r - N)'$.
Note that there are simply-connected RDPs in positive characteristics.

\begin{table}
	\caption{Fundamental groups and the universal coverings of RDPs} \label{table:pi1}
	\begin{tabular}{lllll}
\toprule
char & univ.\ cov. & RDP & & $\pi_1$ \\
\midrule
any          & $A_{p^e-1}$ & $A_{n p^e - 1}$ & ($p \notdivides n$) & $C_{n}$: cyclic (of order $n$) \\
$\neq 2$     & $A_{p^e-1}$ & $D_{n p^e + 2}$ & ($p \notdivides n$) & $\BinDih_{n}$: binary dihedral (of order $4 n$) \\
$\neq 2,3$   & smooth      & $E_6$           &                     & $\BinTet$: binary tetrahedral (of order $24$) \\
$\neq 2,3$   & smooth      & $E_7$           &                     & $\BinOct$: binary octahedral (of order $48$) \\
$\neq 2,3,5$ & smooth      & $E_8$           &                     & $\BinIco$: binary icosahedral (of order $120$) \\
\midrule
$2$ & $A_{2^{e+1}-1}$ & $D_{N}^r$ & ($2 \divides N$, $4r > N$)  & $\Dih_{(4r-N)'}$, $4r - N = 2^e (4r - N)'$ \\
$2$ & smooth          & $D_{N}^r$ & ($2 \divides N$, $4r = N$)  & $C_2$  \\
$2$ & $D_{N}^r$       & $D_{N}^r$ & ($2 \divides N$, $4r < N$)  & $0$  \\
$2$ & $A_{1}$         & $D_{N}^r$ & ($2 \notdivides N$, $4r + 2 > N$)  & \small{$\Dih_{4r+2-N}$: dihedral (of order $2(4r+2-N)$)} \\
$2$ & $D_{N}^r$       & $D_{N}^r$ & ($2 \notdivides N$, $4r + 2 < N$)  & $0$  \\
$2$ & $D_4^0$         & $E_6^0$   &             & $C_3$ \\
$2$ & smooth          & $E_6^1$   &             & $C_6$ \\
$2$ & $E_7^r$         & $E_7^r$   & ($r=0,1,2$) & $0$ \\
$2$ & smooth          & $E_7^3$   &             & $C_4$ \\
$2$ & $E_8^r$         & $E_8^r$   & ($r=0,1,3$) & $0$ \\
$2$ & smooth          & $E_8^2$   &             & $C_2$ \\
$2$ & smooth          & $E_8^4$   &             & $C_3 \rtimes C_4$: metacyclic (of order $12$) \\
\midrule
$3$ & $E_6^0$         & $E_6^0$   &             & $0$ \\
$3$ & smooth          & $E_6^1$   &             & $C_3$ \\
$3$ & $E_6^0$         & $E_7^0$   &             & $C_2$ \\
$3$ & smooth          & $E_7^1$   &             & $C_6$ \\
$3$ & $E_8^r$         & $E_8^r$   & ($r = 0,1$) & $0$ \\
$3$ & smooth          & $E_8^2$   &             & $\BinTet$: binary tetrahedral (of order $24$) \\
\midrule
$5$ & $E_8^0$         & $E_8^0$   &             & $0$ \\
$5$ & smooth          & $E_8^1$   &             & $C_5$ \\
\bottomrule
\end{tabular}

\end{table}

Suppose $\Spec C $ admits a symplectic action of $G$.
Then the quotient $(\Spec C)/G = \Spec (C^G)$ is also an RDP by Proposition \ref{prop:tame},
and the universal covering $\Spec \tilde{C}$ of $\Spec C$ and $\Spec C^G$ coincide.
Here $\tilde{C}$ is defined to be the normalization of $C$ in the universal covering of $\Spec C \setminus \set{\fm}$.
It follows that $N := \pi_1(C)$, $\tilde{G} := \pi_1(C^G)$, and $G$ fit into an exact sequence $1 \to N \to \tilde{G} \to G \to 1$ of groups.
Using the classification of $\tilde{G}$ (Table \ref{table:pi1}), we obtain Table \ref{table:tame quotient},
where (*) is $\Dih_n$ or $\Dic_n$ respectively if $m-1$ is odd or even,
and in the latter case $n$ is assumed to be even.
It is assumed that $p \notdivides m$ for $D_{m p^e+2}$ and $A_{m p^e-1}$.
The symbols $D_3$ and $A_0$ are read to be $A_3$ and smooth respectively.

\begin{table}
\caption{Tame quotient morphisms between RDPs} \label{table:tame quotient}
\begin{tabular}{l|lll|lll}
	\toprule
	char & $N = \pi_1(C)$ & $\tilde{G} = \pi_1(C^G)$ & $ G$ & $\tilde{C}$ & $C$ & $C^G$ \\ 
	\midrule
	$\neq 2,3$ & $\BinTet      $ & $\BinOct       $   & $C_2$   & smooth      & $E_6$   & $E_7$      \\
	$3$        & $1$             & $C_2$              & $C_2$   & $E_6^0$     & $E_6^0$ & $E_7^0$    \\
	$3$        & $C_3$           & $C_6$              & $C_2$   & smooth      & $E_6^1$ & $E_7^1$    \\
	$\neq 2$   & $\BinDih_{m}$   & $\BinDih_{2m}$     & $C_2$   & $A_{p^e-1}$ & $D_{m p^e+2}$ & $D_{2 m p^e+2}$ \\
	$\neq 2,3$ & $\BinDih_{2}  $ & $\BinOct       $   & $\fS_3$ & smooth      & $D_4$   & $E_7$      \\
	$\neq 2,3$ & $\BinDih_{2}  $ & $\BinTet         $ & $\fA_3$ & smooth      & $D_4$   & $E_6$      \\
	$2$        & $1$             & $C_3$              & $\fA_3$ & $D_4^0$     & $D_4^0$ & $E_6^0$    \\
	$2$        & $C_2$           & $C_6$              & $\fA_3$ & smooth      & $D_4^1$ & $E_6^1$    \\
	$\neq 2$   & $C_{m}      $   & $\BinDih_{nm/2} $  & (*)     & $A_{p^e-1}$ & $A_{mp^e-1}$ & $D_{nmp^e/2+2}$ \\
	any        & $C_{m}        $ & $C_{nm}         $  & $C_n$   & $A_{p^e-1}$ & $A_{mp^e-1}$ & $A_{nmp^e-1}$   \\
	$\neq 2,3$ & $\{ \pm 1 \}  $ & $\BinTet, \BinOct$ & $\Tet, \Oct$ & smooth & $A_1$   & $E_6, E_7$ \\
	$\neq 2,3,5$ & $\{ \pm 1 \}$ & $\BinIco$          & $\Ico$  & smooth      & $A_1$   & $E_8$      \\
	\bottomrule
\end{tabular}
\end{table}

It remains to observe that in each case the $G'$-action on $B'_0$ as described in Proposition \ref{prop:local action} gives the desired quotient singularity,
which is straightforward.
\end{proof}

\begin{thm} \label{thm:versal G-deformation}
Let $\map{\rho}{G}{\GL_n(W(k))}$ be a representation of a finite group $G$ of order prime to $\charac(k)$,
and write $\rho(g) = (\rho(g)_{ij})_{i,j=1}^n$.
Extend the action to $\map{\rho}{G}{\Aut(W(k)[[x_1, \dots, x_n]])}$.
Let $\map{c}{G}{W(k)^*}$ be a character, and denote by $V^{G = c}$ the eigenspace of a representation $V$.
Let $F \in W(k)[x_1, \dots, x_n]^{G = c}$ and $X := (\bar{F} = 0) \subset \hat{\bA}^n_k$.

Define $W(k)[[x_1, \dots, x_n]]^G$-modules $M$ and $T^{1,c}$ by 
\begin{align*}
M &:= \set{(h_i) \in W(k)[[x_1, \dots, x_n]]^{\oplus n} \mid \; \rho(g)(h_i) = \sum_{j} \rho(g)_{ij} h_j},  \\
T^{1,c} &:= W(k)[[x_1, \dots, x_n]]^{G = c} / (F \cdot W(k)[[x_1, \dots, x_n]]^{G} + (F_{x_i}) \cdot M), 
\end{align*}
where $(F_{x_i}) \cdot M := \set{\sum_i F_{x_i} h_i \mid (h_i) \in M}$.
Suppose $e_1, \dots, e_{\tau}$ generate the $W(k)$-module $T^{1,c}$.
Then $\cX_S := ({F} + \sum_{j = 1}^{\tau} s_j {e}_j = 0) \subset \hat{\bA}^{n+\tau}_{W(k)} \to S := \hat{\bA}^{\tau}_{W(k)}$ 
is a $G$-equivariant versal deformation of $X$ in the following sense:

Suppose $\cX \to S'$ is a deformation of $X$ over a complete local affine $W(k)$-scheme $S'$ with residue field $k$, 
equipped with an action $G \to \Aut_{S'}(\cX)$ compatible with the action on $X$.
Then $\cX$ is isomorphic to the pullback of $\cX_S$ by some morphism $S' \to S$.
\end{thm}

\begin{rem} \label{rem:finite generation of T}
If $G = 1$ and $c = 1$, then $(F_{x_i}) \cdot M$ is simply the ideal generated by $F_{x_i}$, 
and $T^{1,c} \otimes k = k[[x_1, \dots, x_n]] / (F, F_{x_1}, \dots, F_{x_n})$ is the usual Tjurina algebra.

In general $T^{1,c}$ may not be a finitely generated $W(k)$-module,
even if $X$ is an isolated singularity.
For example, suppose $p \neq 2,3$ and let 
\[ 
 G = \biggl( \biggl\{ \begin{pmatrix} \pm 1 & & \\ & \pm 1 & \\ & & \pm 1 \end{pmatrix} \biggr\} \cap \SL_3 \biggr) \rtimes 
 \biggl\langle \begin{pmatrix} & 1 & \\ & & 1 \\ 1 & & \end{pmatrix} \biggr\rangle
\]
act on $W(k)[[x_1, \dots, x_3]]$ linearly,
and let $c = 1$, $F = x_1^2 + x_2^2 + x_3^2$.
(This is the case of $(\Tet, A_1)$.)
Then $W(k)[[x_1, \dots, x_3]]^G = W(k)[[A, B, C, \delta]] / 
(- \delta^2 + (-4 A^3 C^2 + A^2 B^2 + 18 A B C^2 - 4 B^3 - 27 C^4)$,
where $A = x_1^2 + x_2^2 + x_3^2$, $B = x_2^2 x_3^2 + x_3^2 x_1^2 + x_1^2 x_2^2$, 
$C = x_1 x_2 x_3$, and $\delta = (x_1^2 - x_2^2)(x_2^2 - x_3^2)(x_3^2 - x_1^2)$.
We have $M = \spanned{(x_1, x_2, x_3), (x_2 x_3, x_3 x_1, x_1 x_2)}$, and hence 
$T^{1,c} = W(k)[[A, B, C, \delta]] / (-\delta^2 + (\dots), A, C) \cong W(k)[[B, \delta]] / (-\delta^2 - 4 B^3)$
is not finitely generated as a $W(k)$-module.
\end{rem}

\begin{proof}
The proof is parallel to the one given in \cite{Greuel--Lossen--Shustin:deformations}*{proof of Theorem II.1.16} (which deals with deformations over $\bC$ without a group action).
We may assume that $\cX$ is embedded, i.e.\ $\cX = (\tilde{F} = 0) \subset \hat{\bA}^{n}_{W(k)} \times S'$, 
$S' = \Spec R' \subset \bA^{r}_{W(k)}$,
with $\tilde{F} \otimes_{R'} k = \bar{F}$.
We will find 
$\phi = (\phi_j) \in R'^{\tau}$,
$(h_i) \in M \otimes R'$,
$H \in W(k)[[x_1, \dots, x_n]]^{G=1} \otimes R'$
satisfying 
$\phi_j \otimes_{R'} k = 0$,
$h_i \otimes_{R'} k = x_i$, 
$H \otimes_{R'} k = 0$, 
and $\tilde{F}(h_i) = (1 + H) (F + \sum \phi_j e_j)$.
Then the conditions imply that
$\cX \cong \cX_{S} \times_{S} S'$ via the $G$-equivariant morphism defined by $h_i$, 
where the morphism $S' \to S$ is defined by $\phi$.
We construct such elements modulo $\idealm_{R'}^{l}$ by induction on $l \geq 1$. 
For $l = 1$ we take $\phi = 0$, $h_i = x_i$, $H = 0$.
At each induction step we have to, 
for a certain element $\xi \in W(k)[x_1, \dots, x_n]^{G = c}$,
find $\phi^{(l)}, h_i^{(l)}, H^{(l)}$ satisfying 
$\xi = \sum_{j = 1}^{\tau} e_j \phi_{j}^{(l)} - (F_{x_i}) \cdot (h_{i}^{(l)}) + F H^{(l)}$,
which is indeed possible from the definition of $e_j$.
\end{proof}

\begin{proof}[Proof of Proposition \ref{prop:local action}]
By Proposition \ref{prop:symplecticRDP}, the pair $(G, B \otimes k)$ is as in the list of Proposition \ref{prop:local action}.
For the cases $(G, A_1)$ with $G = \Tet, \Oct, \Ico$, we have nothing to prove.
Consider the other cases.
Since the assertion is \'etale local, we replace $B$ with its Henselization.

We first reduce the proposition to showing that the completion $\hat{B}$ (with respect to the maximal ideal) 
is of the form $\cO_K[[x,y,z]]/(F)$ with $F$ and the $G$-action as in the statement.
Suppose $\hat{B}$ is of this form. 
As in Theorem \ref{thm:versal G-deformation}, we define 
\[
M := \set{(h_i) \in B^{\oplus n} \mid \; \rho(g)(h_i) = \sum_{j} \rho(g)_{ij} h_j}
\]
with respect to the action $\rho$ as in the statement.
It suffices to find a coordinate $x', y', z' \in B$ satisfying $(x',y',z') \in M$ and $F(x', y', z') = 0$. 
It is easy to find a coordinate $x'', y'', z'' \in B$ satisfying $(x'',y'',z'') \in M$. 
We observe that $M$ is a finitely generated $B^G$-module.
(For example, if $G = C_n$ acts by $(x'', y'', z'') \mapsto (\zeta_n x'', \zeta_n^{-1} y'', z'')$,
then $M$ is generated by $(x'', y'', 0), (y''^{n-1}, x''^{n-1}, 0), (0, 0, 1)$.)
Hence the problem is reduced to a system of polynomial equations on $B^G$.
Since there is a solution in $\hat{B}^G$, we obtain a solution in $B^G$ by Lemma \ref{lem:approximation}.

Now consider $\hat{B}$.
We use Theorem \ref{thm:versal G-deformation}.
Let $F \in W(k)[x,y,z]$ be 
 as in the statement of Proposition \ref{prop:local action}, with all $q_l = 0$.
Define the $G$-action on $W(k)[[x,y,z]]$ as in the statement.
Let $c$ be the quadratic character with kernel $\spanned{\sigma}$ if $G = \Dic_{n}$
and the trivial character if otherwise, so that $F \in W(k)[[x,y,z]]^{G = c}$.
It remains to find generators of the module $T^{1,c}$.
Let us explain the case of $(\Dic_{n}, A_{m-1})$ with $m$ odd and $n$ even (other cases are easier).
We have $T^{1,c} = \spanned{x^{n} - y^{n}, z}$
as a $W(k)[[x^{n}+y^{n}, (x^{n}-y^{n})z, z^2]]/(\dots)$-module.
For the elements $(h_i) = (h_1,h_2,h_3) = (x,y,0), (y^{n-1}, -x^{n-1},0) \in M$,
we have $(F_{x_i}) \cdot (h_i) = -2z^m, y^{n}-x^{n}$ respectively.
Hence $z, z^3, \dots, z^{m-2}$ generates $T^{1,c}$.

The alternative forms in the cases $(\fS_3, D_4)$ and $(\fA_3, D_4)$,
are obtained as follows:
We write $F = x^2 + y^3 + z^3 -3yzA(x) + C(x)$ with $A, C \in \cO_K[x]$, 
we let $y_i = \zeta_3^{i} y + \zeta_3^{-i} z + A(x)$,
and then we have $x^2 + y_1 y_2 y_3 + (C(x) - A(x)^3) = y_1 + y_2 + y_3 - 3 A(x) = 0$.
\end{proof}

\begin{prop} \label{prop:local resolution}
	Let $B$ and $G$ be as in Proposition \ref{prop:local action}.
	Then, after replacing $K$ by a finite extension,
	$B$ admits a $G$-equivariant simultaneous resolution.
\end{prop}

\begin{proof}
	If $G = 1$, this is \cite{Artin:brieskorn}*{Theorem 2}. Suppose $G \neq 1$.\
	
	We first show that it suffices to give a simultaneous $G$-resolution after an \'etale base change.
	Indeed, assume that $B \to B_1$ is a local \'etale $G$-equivariant homomorphism
	and $f \colon X \to \Spec B_1$ is a simultaneous $G$-resolution.
	By extending $K$ we may assume that $B/\fm \to B_1/\fm_1$ is an isomorphism.
	Let $V = \Spec B$, $o \in V$ the closed point, and $V^* = V \setminus \{ o \}$.
	Define $V_1,o_1,V_1^*$ similarly.
	Write $R = V_1 \times_V V_1$, which is the \'etale equivalence relation on $V_1$ inducing $V = V_1 / R$.
	Then we have $R = \Delta(V_1) \sqcup R^*$, where $\Delta$ is the diagonal,
	and $R^* \subset V_1^* \times_{V^*} V_1^*$.
	Now let $R' = \Delta(X) \sqcup f^{*}(R^*) \subset X \times_V X$.
	Here $f^*(R^*)$ is isomorphic to $R^*$ since $f$ is an isomorphism over $V_1^*$.
	Then $R'$ is an \'etale equivalence relation on $X$ and 
	$X / R' \to V_1 / R = V$ is a simultaneous $G$-resolution.
	
	Thus it suffices to give a simultaneous resolution of $B'$ as in Proposition \ref{prop:local action},
	and we write simply $B$ in place of $B'$. 
	
	\medskip

	Suppose $\Sing(B)$ is $A_1$.
	
	The local Picard group $\Cl(B)$ of $B$ is isomorphic to $\bZ$
	(since $B \cong \cO_K[x,y,z]^{\hens}/(xy + z^2 + q_1 z + q_0)$ for some $q_1, q_0 \in \fp$).
	Let $I_+$ and $I_-$ be ideals of Weil divisors 
	that are the two generators of $\Cl(B)$.
	We will show that each of the two blow-ups at $I_+$ and $I_-$ is a $G$-resolution.
	To show this it suffices to check that it is a $H$-resolution for each cyclic subgroup $H \subset G$. Thus we may assume that $G$ is cyclic.
	We conclude as in the next case.
	
	(Applying Shepherd-Barron's result (see Remark \ref{rem:localresolution})
	to the case of $A_1$, it follows that there are no other resolution,
	hence \emph{any} simultaneous resolution is $G$-equivariant.)

\medskip

	(Case $(C_n, A_{m-1})$ ($m \geq 2$)):
	By replacing $K$ by a finite extension,
	we obtain $F = xy + \prod_{i = 1}^{m} (z-\alpha_i)$ for some $\alpha_i \in \fp$.
	(Since the generic fiber is smooth it follows that $\alpha_i$'s are distinct.)
	Let $I_j = (x, \prod_{i = 1}^j (z - \alpha_i) )$ ($j = 1, \ldots, m-1$).
	Then these ideals are $G$-invariant and 
	the blow-up at the ideal $I = I_1 I_2 \cdots I_{m-1}$ is a simultaneous $G$-resolution.
	
	\medskip
	
	(Cases $(\Dih_n, A_{m-1})$ ($m \geq 2$ even) and $(\Dic_n, A_{m-1})$ ($m \geq 3$ odd, $n$ even)):
	By replacing $K$ by a finite extension,
	we obtain $F = xy + \prod_{i = 1}^{m} (z-\alpha_i)$ 
	for some $\alpha_i \in \fp$
	satisfying $\alpha_{m+1 -i} = - \alpha_i$
	(hence $\alpha_{(m+1)/2} = 0$ if $m$ is odd).
	Define $I_j$ as in the previous case.
	Then, because of the identity $xy = - \prod (z - \alpha_i)$,
	the blow-up at $\tau(I_j) = (y, (z - \alpha_{m+1-j}) \cdots (z - \alpha_{m-1}) (z - \alpha_{m}))$
	coincides with the blow-up at $I_{m-j} = (x, (z - \alpha_1) (z - \alpha_2) \cdots (z - \alpha_{m-j}))$.
	This shows that the blow-up at $I_j I_{m-j}$ is $\tau$-equivariant (even though the ideal itself is not $\tau$-stable).
	Likewise, 
	the blow-up at $I = \prod I_j$ is $\tau$-equivariant and hence is a simultaneous $G$-resolution.

	\medskip
	
		For each remaining case, it suffices to give a partial simultaneous $G'$-resolution of $B'$. 
Here, 
we define a \emph{partial (simultaneous) resolution} of a local ring $B$ as in Theorem \ref{thm:localresolution}
to be a proper morphism $f \colon X \to \Spec B$ from an algebraic space $X$
such that, $f$ is an isomorphism on the generic fiber,
$f$ is not an isomorphism on the special fiber,
all singularity of $X_0$ are RDPs (if any),
and the minimal resolution of $X_0$ is the minimal resolution of $\Spec B_0$ ($B_0 = B \otimes k$).
It follows that $X_0$ has less RDPs than $\Spec B_0$
(when $A_n,D_n,E_n$ are counted with weight $n$).

	\medskip
	
	(Case $(C_2, D_m)$ ($m \geq 4$)): 
	Write $y^{m-1} + \sum_{l = 0}^{m-2} q_l y^l = -(A(y)^2 + y C(y)^2)$
	with polynomials $A, C \in \cO_K[y]$. 
	(To find such $A,C$, we write $y^{m-1} + \sum q_l y^l = \prod (y + \beta_i^2)$, 
	and write $\prod (\beta_i + \sqrt{-y}) = \sqrt{-1} (A + C \sqrt{-y})$ in $\cO_K[\sqrt{-y}]$ with $A,C \in \cO_K[y]$).
	Then we have $F = (x + A)(x - A) + y (z + C)(z - C)$
	and the ideal $I = (x + A, z + C)(x - A, z - C)$ is $G$-invariant.
	The blow-up at $I$ is a partial $G$-resolution,
	whose special fiber has a single singularity, of type $A_{m-2}$.

	\medskip
	
	(Cases $(\fS_3, D_4)$ and $(\fA_3, D_4)$) 
	($p$ may be $= 2$):
	We use the alternative form of Proposition \ref{prop:local action}:
	$B = \cO_K[x,y_1,y_2,y_3]^{\hens} / (y_1 y_2 y_3 + Q(x), y_1 + y_2 + y_3 - R(x))$.
	
	Write $R(x) = r_1 x + r_0$ ($r_0 \in \fp$.
	Take the decomposition $Q(x) = (h_1 x + h_0) (a_1 x + a_0) (b_1 x + b_0)$
	with $a_1,b_1,h_0 \in \cO_K^*$, $a_0,b_0 \in \fp$, and $h_1 \in \cO_K$.
	We have $a_1 b_0 - a_0 b_1 \neq 0$, since otherwise the generic fiber has singularity.
	Write $H(x) = h_1 x + h_0$. 
	Take nonzero $\gamma,\delta \in \fp$ satisfying 
	$ \gamma b_j + \delta a_j + \gamma \delta r_j + (\gamma \delta)^2 h_j = 0$
	for $j = 0,1$: 
	the existence of a solution follows from a straightforward argument using the conditions on the coefficients. 
	If $p \neq 2$ then 
	(we have $r_1 = 0$ and) we moreover assume $H(x) = 1$, $a_1 = b_1 = 1$, $a_0 = - b_0$,
	and then we have $\gamma = - \delta$.
	Then we have 
	\begin{align*}
	F_1 &= 
	H(x) (a_1 x + a_0 + \gamma y_i) (b_1 x + b_0 + \delta y_i) \\
	&\quad {}  + y_i (y_{i+1} + \gamma \delta H(x)) (y_{i+2} + \gamma \delta H(x)) + \varepsilon_i 
	\end{align*}
	in $\cO_K \polyHens{x,y_1,y_2,y_3}$,
	where 
	\begin{align*}
	\varepsilon_i 
	&= - H(x) y_i (((b_1 x + b_0) \gamma + (a_1 x + a_0) \delta +\gamma \delta R(x) + (\gamma \delta)^2 H(x)) + \gamma \delta F_2) \\
	&= - \gamma \delta H(x) y_i F_2
	\in (F_2).
	\end{align*}
	Let $I_i = (a_1 x + a_0 + \gamma y_{i-1}, y_{i} + \gamma \delta H(x)) \subset B$.
	Then we have $\rho(I_i) = I_{\rho(i)}$ for each $\rho \in G \subset \fS_3$.
	Indeed, clearly $(123) I_i = I_{i+1}$ and,
	if $G = \fS_3$ (in which case $p \neq 2$),
	$(i,i+1) I_i = I_{i+1}$ follows from the equality
	\begin{align*}
	- a_1 x + a_0 + \gamma y_{i-1}
	&= - (a_1 x + a_0 + \gamma y_i)
	- \gamma (y_{i+1} + \gamma \delta H) \\
	& \qquad 
	+ \gamma F_2
	+ (2 a_0 + \gamma^2 \delta H + \gamma R) \\
	&\equiv - (a_1 x + a_0 + \gamma y_i)
	- \gamma (y_{i+1} + \gamma \delta H)
	\pmod{(F_2)}
	\end{align*}
	in $\cO_K \polyHens{x,y_1,y_2,y_3}$,
	as we have $2 a_0 + \gamma^2 \delta H + \gamma R = 0$ by the conditions on $a_i,b_i,h_i,r_i$ and $\gamma,\delta$.
	Hence the ideal $J = I_1 I_2 I_3$ is $G$-invariant.
	The blow-up at $J$ is a partial $G$-resolution,
	whose special fiber has a single singularity, of type $A_{1}$.

	\medskip
	
	(Case $(C_2, E_6)$) 
	($p$ may be $= 3$):
	We can write $F = x^2 - (z^2 - H(y))^2 + 4 T(y)$
	with $H = \sum_{i = 0}^{2} h_i y^i$
	and $T = \sum_{i = 0}^{4} t_i y^i$
	with $h_0,h_1,t_0,t_1,t_2 \in \fp$, $t_3 \in \cO_K^*$, $h_2,t_4 \in \cO_K$.
	Take a decomposition $T = RS$ with $R,S \in \cO_K[y]$ with $\deg R = 1,2$, $\deg S = 2$, 
	$\ord_y(R \bmod \fp) = 1$, and $\ord_y(S \bmod \fp) = 2$.
	Write $R = \sum_{i=0}^{2} r_i y^i$ and $S = \sum_{i=0}^{2} s_i y^i$.
	We find $A \in \cO_K[y]$ (of degree $\leq 2$), $b,c_0 \in \fp$ and $c_1 \in \cO_K^*$ satisfying, letting $C(y) = c_1 y + c_0$,
	\begin{align*}
	H &= -A + 2 b^2 R\\
	-H^2 + 4 T &= -A^2 - 4 R C^2
	\end{align*}
	so that $F = (x + z^2 + A)(x - z^2 - A) + 4 R(bz + C)(bz - C)$.
	Then the blowup at the ($G$-invariant) ideal $(x + z^2 + A, bz + C)(x - z^2 - A, bz - C)$ is a partial simultaneous $G$-resolution,
	whose special fiber has a single singularity, of type $D_4$.
	By eliminating $A$, we need 
	$b^4 R - b^2 H + S = -C^2$.
	For the left hand side to be a square we need 
	\[
	(r_1 b^4 - h_1 b^2 + s_1)^2 - (r_0 b^4 - h_0 b^2 + s_0)(r_2 b^4 - h_2 b^2 + s_2) = 0,
	\]
	which indeed has solution $b$ in $\fp$ since $r_0,s_0,s_1,h_0,h_1 \in \fp$ and $r_1 \in \cO_K^*$.
\end{proof}

\begin{proof}[Proof of Theorem \ref{thm:localresolution}]
If $G$ is symplectic then this follows from Proposition \ref{prop:local resolution} inductively.

\medskip

Now assume $G$ is non-symplectic.
We may assume that $G$ is cyclic with generator $g$.

First we reduce to the special case of $A_1$ or $A_2$ and $G$ acting on the exceptional curves transitively.
Assume we have a $G$-resolution $\pi \colon \cX \to \cX'$ and let $E$ be the exceptional divisor.
Then, by the shape of the Dynkin diagram, 
the set of components of $E$ 
has a $G$-orbit $O$ consisting of one or two adjacent elements.
Then $\pi$ factors through a $G$-equivariant morphism $\pi'' \colon \cX \to \cX''$
that contracts exactly components in $O$ (as in the proof of \cite{Liedtke--Matsumoto}*{Proposition 4.2}).
Such $\pi''$, which gives a $G$-equivariant simultaneous resolution of $A_1$ or $A_2$,
cannot exist according to the special case.

Consider the special case. 
It suffices to show that the completion $\hat{B}$ of $B$ does not admit a $G$-equivariant simultaneous resolution.
For simplicity we write $B$ in place of $\hat{B}$.
Assume $\pi \colon \cX \to \Spec B$
is a $G$-resolution.
Let $E_1, \ldots, E_m$ be the exceptional curves ($m = 1,2$).
Then $\pi$ induces a $G$-equivariant homomorphism
$(R^1 \pi_* \cO^*_{\cX})_{\bar x} \to \Cl (B)$
where $\bar x$ is the geometric point of $\Spec B$ above the maximal ideal,
and $\Cl(B)$ is the local Picard group.
This map is surjective since,
for each \'etale neighborhood $V$ of $\bar x$,
the group $\Cl(\cO(V))$ is generated by classes of Weil divisors $D$ on $V$ 
and we can take $\cO(\pi^{-1}(D)) \in \Pic(\pi^{-1}(V))$ as their inverse images.
Since the source is generated by the classes of $E_1, \ldots, E_m$, 
the $G$-action on it factors through a group of order $m!$,
and if $m = 2$ its eigenvalue $-1$ has multiplicity $1$. 
It suffices to check that the $G$-action on $\Cl(B)$
is not a quotient of this type.

We will give a normal form $B \cong \cO_K[[x,y,z]] / (F)$. 
We may assume that the generator $g$ acts by $g(x,y,z) = (ax,by,cz)$ and that $F \in \cO_K[[x,y,z]]$ satisfies $g(F) = eF$.
Since the action is non-symplectic we have $e \neq abc$.
Let $\bar{F}_2$ be the degree $2$ part of $\bar{F} = (F \bmod \idealp)$.
We may assume that $\bar{F}_2 = xy + z^2$ (resp.\ $\bar{F}_2 = xy$ or $\bar{F}_2 = x^2 - y^2$) in the case of $A_1$ (resp.\ $A_2$).
Indeed, if $xy,yz,zx$ do not appear in $\bar{F}_2$ then exactly three (resp.\ two) of $x^2,y^2,z^2$ appear, 
and then by a coordinate change we obtain the desired form.
Then by an argument similar to the proof of Theorem \ref{thm:versal G-deformation}, we may assume that $\bar{F} = xy + z^2$ (resp.\ $\bar{F} = xy + z^3$).
Then by Theorem \ref{thm:versal G-deformation} we obtain 
$F = xy + z^2 + q_1 z + q_0$ (resp.\ $F = xy + z^3 + q_2 z^2 + q_1 z + q_0$),
and some of $q_l$ (those not compatible with the $G$-action) are automatically zero.
Since the generic fiber is non-singular, at least one of $q_1$ and $q_0$ should be nonzero.
Hence we may assume that $F$ and the $G$-action are one of the following,
where the first case is $A_1$ and the others are $A_2$:
\begin{itemize}
\item $g(x,y,z) = (ax, a^{-1}y, -z)$, $F = xy + z^2 + q_0$.
\item $g(x,y,z) = (ax, -a^{-1}y, -z)$, $F = xy + z^3 + q_1 z$. 
\item $g(x,y,z) = (ax, a^{-1}y, \zeta_3 z)$, $F = xy + z^3 + q_0$. 
\item $g(x,y,z) = (x, -y, z)$, $F = x^2 - y^2 + z^3 + q_2 z^2 + q_1 z + q_0$.
\item $g(x,y,z) = (x, -y, \zeta_3 z)$, $F = x^2 - y^2 + z^3 + q_0$.
\end{itemize}

Consider the first case ($A_1$).
Since $\Cl(B)$ is an infinite cyclic group 
generated by $[D_+] = -[D_-]$,
where $D_{\pm} = (x = z \pm \sqrt{-q_0} = 0)$,
$g$ acts on $\Cl(B)$ by $-1$
(cf.\ \cite{Liedtke--Matsumoto}*{Section 7}). 
Hence $\Cl(B)$ cannot be the image of $(R^1 \pi_* \cO^*_{\cX})_{\bar x}$.

Consider the other cases ($A_2$). Only in the latter two cases $g$ swaps $E_1$ and $E_2$.
To compute the action on $\Cl (B)$,
we can use the generators $[D^+_i], [D^-_i]$ ($i = 1,2,3$),
subject to relations $[D^+_i] + [D^-_i] = \sum [D^+_i] = \sum [D^-_i] = 0$,
defined by $D^+_i = (x + y, z - \alpha_i)$, $D^-_i = (x - y, z - \alpha_i)$
where $\prod (z - \alpha_i) = z^3 + \dots + q_0$ is the decomposition.
In the the fourth case the action of $g$ on $\Cl (B)$ is of order $6$.
In the third case, the action of $g$ on $\Cl (B)$ is of order $2$
but its eigenvalue $-1$ has multiplicity $2$.
Hence $\Cl(B)$ cannot be the image of $(R^1 \pi_* \cO^*_{\cX})_{\bar x}$.
\end{proof}

\section{$G$-equivariant flops} \label{sec:G-flops}

In this section we prove the existence and termination of $G$-equivariant flops for $G$-models of K3 surfaces
(more generally surfaces with numerically trivial canonical divisor), 
relying on the results in our previous paper \cite{Liedtke--Matsumoto}*{Section 4}.

\subsection{Results of Liedtke--Matsumoto}

In this subsection we recall
the results of 
\cite{Liedtke--Matsumoto}*{Section 3} 
on the existence and termination of flops 
between proper smooth models of a fixed K3 surface.

The following definitions, taken from \cite{Liedtke--Matsumoto}*{Section 4},
are adjustments of those in \cite{Kollar--Mori}*{Definitions 3.33 and 6.10} to our situation of models of surfaces.

\begin{defn}
  Let $X$ be a smooth and proper surface over $K$ with numerically 
  trivial $\omega_{X/K}$ that has a proper smooth model $\cX\to\Spec\cO_K$.
  Then,
  \begin{enumerate}
    \item A proper and birational morphism $f \colon \cX \to \cY$ over $\cO_K$ 
      is called a \emph{flopping contraction}
      if $\cY$ is normal, $\omega_{\cX/\cO_K}$ is numerically $f$-trivial, and 
      the exceptional locus of $f$ is of codimension at least $2$.
    \item If $D$ is a Cartier divisor on $\cX$, then a
      birational map $\cX \rationalto \cX^+$ over $\cO_K$
      is called a $D$-\emph{flop}
      if it decomposes into a flopping contraction $f \colon \cX \to \cY$ 
      followed by (the inverse of) a flopping contraction $f^+ \colon \cX^+ \to \cY$
      such that $-D$ is $f$-ample and $D^+$ is $f^+$-ample,
      where $D^+$ denotes the strict transform of $D$ on $\cX^+$.
   \item A morphism $f^+$ as in (2) is also called a \emph{flop} of $f$.
  \end{enumerate} 
\end{defn}
A flop of $f$, if exists, does not depend on the choice of $D$ by
\cite{Kollar--Mori}*{Corollary 6.4, Definition 6.10}.
This justifies talking about flops without referring to $D$.

In \cite{Liedtke--Matsumoto}*{Section 4} we proved that:

\begin{prop}[existence and termination of flops, \cite{Liedtke--Matsumoto}*{Propositions 4.2 and 4.5}] \label{prop:flops}
Let $X$ be a surface over $K$ with numerically trivial canonical divisor, and $\cY$ a proper smooth model of $X$ over $\cO_K$.
Let $\cL$ be an ample line bundle on $X$, 
and denote by $\cL_0$ the restriction to $\cY_0$ of the extension to $\cY$ of $\cL$.
Then we have the following. 

\begin{enumerate}
\item \label{LM:existence of flop}
Let $Z = \bigcup C_i$ be a union of finitely many $\cL_0$-negative integral curves $C_i$.
Then we have a flopping contraction $f \colon \cY \to \cY'$ contracting $C_i$'s and no other curves,
 and we have its flop $\cY \rationalto \cY^+$ over $\cO_K$.
$\cY^+$ is again a proper smooth model of $X$ over $\cO_K$.
\item \label{LM:termination of flop}
After applying finitely many flops as in (\ref{LM:existence of flop}), 
we arrive at a proper smooth model $\cY^\dagger$ of $X$ such that $\cL^\dagger_0$ is nef.
\end{enumerate}
\end{prop}

\begin{rem}
(i)
As showed in the proof of \cite{Liedtke--Matsumoto}*{Proposition 4.2},
there are only finitely many $\cL_0$-negative curves,
and over $\overline k$ those curves are smooth rational curves forming finitely many ADE configurations.
In particular the irreducible components of $Z_{\overline k}$ are again smooth rational curves again forming finitely many ADE configurations.

(ii) 
In \cite{Liedtke--Matsumoto}*{Proposition 4.2}, part (\ref{LM:existence of flop}) is stated only for a single integral (not necessarily geometrically integral) curve $Z$.
But the same proof applies to the case of connected $Z$, 
and we can reduce the general case to the connected case
(since the flop at one connected component of $Z$ does not affect the $\cL_0$-degrees of the curves on the other components).
\end{rem}

We recall another result. 
\begin{prop}[\cite{Liedtke--Matsumoto}*{Proposition 4.6}] \label{prop:projectivemodel}
Let $X$ be a K3 surface over $K$ with good reduction.
Let $\cL$ an ample line bundle of $X$.
Then there exists a projective RDP model $\cX$ of $X$,
the extension of $\cL$ to which is relatively ample.
Such $\cX$ is unique up to isomorphism.
\end{prop}

\subsection{$G$-equivariant flops}

We prove the following $G$-equivariant version.

\begin{prop} \label{prop:G-flops}
Let $X$, $\cY$, $\cL$ as in Proposition \ref{prop:flops}.
Assume $X$ is equipped with an action of a finite group $G$, 
$\cY$ is a $G$-model,
and $\cL$ is $G$-invariant.
\begin{enumerate}
\item \label{LM-G:existence of flop}
 Let $Z$ as in part (\ref{LM:existence of flop}) of Proposition \ref{prop:flops}, 
and assume $Z$ is $G$-stable.
Then $G$ acts canonically on the resulting model $\cY^+$ and the flop is a $G$-equivariant rational map.
\item \label{LM-G:termination of flop}
After applying finitely many flops as in (\ref{LM-G:existence of flop}) , 
we arrive at a proper smooth $G$-model $\cY^\dagger$ of $X$ such that $\cL^\dagger_0$ is nef.
\end{enumerate}
\end{prop}

\begin{proof}
(\ref{LM-G:existence of flop})
This essentially follows from the uniqueness of the flop, as follows.

Giving a $G$-action on $\cY^+$ compatible with that on $X$ is equivalent to giving, 
for each $g \in G$, 
an isomorphism $\cY^+ \isomto g^* \cY^+$ extending the identity $X \isomto X$,
where $g^* \cY^+$ is the normalization of $\cY^+$ in the pullback $g \colon X \to X$.
(It is required that the isomorphisms be compatible with the group structure, 
but once we have morphisms this is automatic since it is trivially true on a dense open subspace $X$.)

Now consider the diagram $\cY \to \cY' \from \cY^+$, the flop at $Z$.
By taking the normalization under the pullback $g \colon X \to X$, we obtain
$ g^* \cY  \to  g^* \cY'  \from  g^* \cY^+ $.
By taking composite with the isomorphism $\cY \isomto g^* \cY$ induced from the $G$-action on $\cY$,
this diagram becomes $\cY \to g^* \cY' \from g^* \cY^+$, the flop at $g^*(Z)$.
Since $g^*(Z) = Z$, the two flopping contractions are the same and the
two flops are the same, hence there are isomorphisms
$ \cY' \isomto g^* \cY' $ and $ \cY^+ \isomto g^* \cY^+$ extending the identity on the generic fiber.

(\ref{LM-G:termination of flop})
Assume $\cL_0$ is not nef, and take an $\cL_0$-negative curve $C$ on $\cY$.
Since $\cL$ is $G$-invariant, images of $C$ under $G$ are all $\cL_0$-negative.
We can apply part (\ref{LM-G:existence of flop}) to the union $Z$ of those images.
Therefore we can conclude from part (\ref{LM:termination of flop}) of Proposition \ref{prop:flops}.
\end{proof}

\begin{prop} \label{prop:G-projectivemodel}
Let $X,\cL$ be as in Proposition \ref{prop:projectivemodel},
$G \subset \Aut(X)$ a subgroup, 
and assume $\cL$ is invariant under $G$.
Then the resulting projective RDP model $\cX$ is naturally a $G$-model.
\end{prop}
\begin{proof}
The uniqueness induces a $G$-action, as in the previous proposition.
\end{proof}

\begin{rem}
This can be applied only to finite $G$, 
since for an ample line bundle $\cL$ on a K3 surface
$\Aut(X,\cL)$ is finite \cite{Huybrechts:lecturesK3}*{Proposition 5.3.3}.
\end{rem}

\section{Proof of the main theorems} \label{sec:proof}

Using the results of Sections \ref{sec:localresolution} and \ref{sec:G-flops}, we can prove Theorem \ref{thm:sympfin}.
We also prove 
Theorem \ref{thm:criterion}(\ref{thm:criterion:symplectic}).

\begin{thm} \label{thm:criterion}
Let $X$ be a (smooth) K3 surface over $K$,
$G$ a finite subgroup of $\Aut(X)$ of order prime to $p$,
and $\cX$ a projective RDP $G$-model of $X$.
\begin{enumerate}
\item \label{thm:criterion:symplectic}
If $G_x = \Stab(x)$ is symplectic for any $x \in \cX^\nonsm$,
then $\cX$ admits a $G$-equivariant simultaneous resolution.
In particular $G$ is extendable.
\item \label{thm:criterion:non-symplectic}
If $G_x$ is non-symplectic for some $x \in \cX^\nonsm$,
then $G$ is not extendable.
\end{enumerate}
\end{thm}

\begin{proof}[Proof of Theorem \ref{thm:sympfin} and Theorem \ref{thm:criterion}(\ref{thm:criterion:symplectic})]
In the case of Theorem \ref{thm:sympfin}, taking a $G$-invariant ample line bundle of $X$ and then applying Proposition \ref{prop:G-projectivemodel}),
we obtain a projective RDP $G$-model $\cX$ (which is in particular a scheme).
In the case of Theorem \ref{thm:criterion}(\ref{thm:criterion:symplectic}), let $\cX$ be as in the statement.
We show that $\cX$ admits a simultaneous $G$-resolution.
By Theorem \ref{thm:localresolution},
for each $x$ in the non-smooth locus $\Sigma = \cX^{\nonsm} \subset \cX$ there is a simultaneous $G_x$-equivariant resolution of $\Spec \cO_{\cX,x}$,
where $G_x = \Stab(x)$.
(Note that the two notions of symplecticness coincide by Lemma \ref{lem:symplectic}(\ref{lem:symplectic:3}).)
We choose a family $(\cY(x) \to \Spec \cO_{\cX,x})_{x \in \Sigma}$ of local simultaneous $G_x$-equivariant resolution
satisfying $g^* \cY(x) = \cY(g^{-1}(x))$.
To show that this is possible,
we consider a $G$-orbit $O$ of $\Sigma$,
take one $x \in O$ and choose one simultaneous $G_x$-resolution $\cY(x)$, 
and then for each other $x' = g^{-1}(x) \in O$ we take $\cY(x')$ to be $g^* \cY(x)$,
which does not depend on the choice of $g$ since $\cY(x)$ is a $G_x$-resolution.
Gluing $\cY(x)$, we obtain a (global) $G$-equivariant simultaneous resolution of $\cX'$.
\end{proof}

\bigskip

Next we consider Theorem \ref{thm:non-extendable}.

As explained in the introduction, we have two methods to prove non-extendability of automorphisms.
In this section we introduce the first one, which uses Theorem \ref{thm:criterion}(\ref{thm:criterion:non-symplectic})
based on birational geometry of $G$-models developed in Section \ref{sec:G-flops},
to prove the case of non-symplectic automorphisms of finite order prime to $p$.

\begin{proof}[Proof of Theorem \ref{thm:criterion}(\ref{thm:criterion:non-symplectic})]
Assume there exists, after extending $K$,
a proper smooth $G$-model $\cY$ of $X$.
Note that then $\omega_{\cY/\cO_K}$ is numerically trivial, as it is trivial on the generic fiber.

Take a relative ample line bundle on $\cX$,
which we may assume to be $G$-invariant.
Then by Propositions \ref{prop:G-flops} and \ref{prop:G-projectivemodel},
we obtain a proper smooth $G$-model $\cY^\dagger$
equipped with a $G$-equivariant morphism $\cY^\dagger \to \cX$.
In other words it is a simultaneous $G$-resolution of $\cX$.
But since $G_x$ is non-symplectic this contradicts Theorem \ref{thm:localresolution}.
\end{proof}

We give examples satisfying assumptions of Theorem \ref{thm:criterion}(\ref{thm:criterion:non-symplectic})
for $p$ arbitrary, $G = \bZ/l\bZ$, $2 \leq l\leq 11$ prime, $l \neq p$.

We fix the notation on elliptic surfaces.
A Weierstrass form 
$F(x,y,t) = y^2 + a_1(t)xy + a_3(t)y + x^3 + a_2(t) x^2 + a_4(t) x + a_6(t) = 0$
over a ring $R$, with $a_i \in R[t]$ with $\deg a_i \leq 2i$,
is considered as a hypersurface of degree $12$ of the weighted projective bundle $\bP(\cO(-4) \oplus \cO(-6) \oplus \cO)$ with weight $4,6,1$ over $\bP^1$.
In particular, $X$ has $\Spec R[x,y,t]/(F)$
and $\Spec R[x',y',s]/(F')$ as open subschemes,
where $F' = y'^2 + a'_1(s)x'y' + a'_3(s)y' + x'^3 + a'_2(s) x'^2 + a'_4(s) x' + a'_6(s)$,
where $a'_i(s) = s^{2i} a_i(1/s) \in R[s]$, with gluing given by $x' = xt^{-4}, y' = yt^{-6}, s = t^{-1}$.
(To cover $X$ by affine schemes we need two more pieces corresponding to $x = y = \infty$ and $x' = y' = \infty$,
but usually they are not important and are omitted.)
If $R$ is a field and these two affine subschemes have only RDP singularities,
then the projective variety is an RDP K3 surface.
If $R = \cO_K$, 
we have a similar criterion for the projective scheme to be an RDP model.

For two primes $p,l$ with $2 \leq l \leq 11$,
we define $X_{l,p}$ and its automorphism $\sigma_{l,p}$ by
\begin{align*}
 X_{11,p} &\colon y^2 + yx + x^3 - (t^{11} - p)            = 0, && y'^2 + s^2 y'x' + x'^3 - s  (1 - p s^{11})            = 0, \\
 X_{ 7,p} &\colon y^2 + yx + x^3 - (t^{ 7} - p)            = 0, && y'^2 + s^2 y'x' + x'^3 - s^5(1 - p s^{ 7})            = 0, \\
 X_{ 5,p} &\colon y^2 + yx + x^3 - (t^{ 5} - p)(t^{5} - 1) = 0, && y'^2 + s^2 y'x' + x'^3 - s^2(1 - p s^{ 5})(1 - s^{5}) = 0, \\
 X_{ 3,p} &\colon y^2 + yx + x^3 - (t^{ 3} - p)(t^{9} - 1) = 0, && y'^2 + s^2 y'x' + x'^3 -    (1 - p s^{ 3})(1 - s^{9}) = 0, \\
 X_{ 2,p} &\colon y^2 + yx + x^3 - (t^{ 2} - p)(t^{8} - 1) = 0, && y'^2 + s^2 y'x' + x'^3 - s^2(1 - p s^{ 2})(1 - s^{8}) = 0, \\
\end{align*}
and $\sigma_{l,p} \colon X_{l,p} \to X_{l,p} \colon (x,y,t) \mapsto (x, y, \zeta_l t), 
(x',y',s) \mapsto (\zeta_l^{-4} x', \zeta_l^{-6} y', \zeta_l^{-1} s)$.
Non-symplecticness is checked by using a global $2$-form
$\omega = (2y + x)^{-1} dx \wedge dt = -(2y' + s^2 x')^{-1} dx' \wedge ds$.
Then the singular points of $X_{l,p}$ in characteristics $0$ and $p$ are as follows
(here, and in the next section, we do not distinguish analytically non-isomorphic RDPs of the same Dynkin diagram):

\begin{tabular}{llll}
	\toprule
$l$ & & char.\ 0 & char.\ $p$ \\
\midrule
each ${ l}$ & $(x,y,t) = (0,0,0)$       & ---    & $A_{l-1}$        \\
 ${ 5,3,2}$ & $(x,y,t) = (0,0,1)$       & ---    & $A_{l^e-1}$ if $p = l$ (*) \\
     ${ 7}$ & $(x',y',s) = (0,0,0)$     & $E_8$  & $E_8$            \\
     ${ 5}$ & $(x',y',s) = (0,0,0)$     & $A_2$  & $A_2$ if $p \neq 2$, 
                                                   $E_7$ if $p = 2$ \\
     ${ 3}$ & $(x',y',s) = (0,1,0)$     & ---    & $D_4$ if $p = 2$ \\
     ${ 2}$ & $(x',y',s) = (0,0,0)$     & $A_2$  & $A_2$ if $p \neq 2$,
                                                  $E_7$ if $p = 2$ \\
\bottomrule
\end{tabular} 

(*) $l^e = 5,9,8$ for $l = 5,3,2$ respectively (this appears in the factor $t^{l^e}-1$ in the formula).

Thus these formula define projective RDP $\sigma$-models $\cX$.
Let $\tilde \cX$ the RDP model obtained as in the first paragraph of the proof of Lemma \ref{lem:resolve rdp}.
This is a projective RDP model.
Moreover, since at each step each RDP on the generic fiber is $\sigma$-fixed,
$\tilde \cX$ admits a natural $\sigma$-action.
Now assume $l \neq p$.
Since the singularity of $\tilde \cX$ at $(x,y,t) = (0,0,0)$ on the special fiber is fixed by $\sigma$, the stabilizer of this point is non-symplectic,
and we can apply Theorem \ref{thm:criterion}(\ref{thm:criterion:non-symplectic}) 
to obtain examples for Theorem \ref{thm:non-extendable} for $G = \bZ/l\bZ$, $2 \leq l \leq 11$, $l \neq p$.

\medskip

We will also give examples which have projective smooth models
for the case $G = \bZ/2\bZ$, $p \neq 2,3$.

Take an integer $a$ satisfying 
$a \equiv 0 \pmod p$ and $a \neq 0$.
Let $F = a^2 z^6 + (x^3 - x z^2)^2 + (y^3 - y z^2)^2$.
Let $\cX$ be the double covering of $\bP^2_{\cO_K}$ defined by $w^2 = F(x,y,z)$.
It is clear that the points defined by $(p = w = x^3-xz^2 = y^3-yz^2 = 0)$ are singular and hence $S = \cX^\nonsm$ contains these points.
A straightforward computation shows that
$\cX$ has no other singular points, 
and that all the points of $S$ are $k$-rational and are RDPs of type $A_1$.

Let $\iota$ be the deck transformation $(w,x,y,z) \mapsto (-w,x,y,z)$.
This defines an involution on $\cX$, and 
all points of $S$ are fixed by $\iota$.
Non-symplecticness of (the restriction $\iota \restrictedto{X}$ to the generic fiber $X$ of) $\iota$ can be showed either by directly computing $(\iota \restrictedto{X})^*(\omega)$ 
for a global $2$-form $\omega = w^{-1} xyz d\log(y/x) \wedge d\log(z/x)$, 
or by checking that $\Fix(\iota \restrictedto{X}) = (w = 0)$ is $1$-dimensional (use Lemma \ref{lem:sympfin}).
By Theorem \ref{thm:criterion}(\ref{thm:criterion:non-symplectic}), $\iota$ is not extendable.

The Weil divisors $\cC_+$ and $\cC_-$ defined by $\cC_{\pm} = (w \pm a z^3 = x^3 - x z^2 + y^3 - y z^2 = 0)$
are non-Cartier exactly at $S$, 
and it can be easily seen that $\Bl_{\cC_+} \cX$ and $\Bl_{\cC_-} \cX$ are projective smooth models of $\cX$.
(Since $\iota$ interchanges $\cC_+$ and $\cC_-$ and the two blow-ups are not isomorphic, 
these smooth models are not $\iota$-models.)

\medskip

The second method of proving non-extendability is to use Proposition \ref{prop:criterion for non-ext} and Corollary \ref{cor:criterion for non-ext}(\ref{item:specialization is trivial}).

In Section \ref{subsec:ker-p-ns} (resp.\ \ref{subsec:ker-p-s})
we give examples, for $2 \leq p \leq 19$ (resp.\ $2 \leq p \leq 7$), 
of non-symplectic (resp.\ symplectic) automorphisms of order $p$ specializing to the identity on the characteristic $p$ fiber.
In Section \ref{subsec:ker-infinite}
we give examples, for $p \geq 2$, of (symplectic and non-symplectic) infinite order automorphisms specializing to the identity.
Together with Corollary \ref{cor:criterion for non-ext}(\ref{item:specialization is trivial}) 
these examples prove the remaining cases of Theorem \ref{thm:non-extendable}.

\section{Automorphisms specializing to identity} \label{sec:kersp}

\subsection{Restriction on the residue characteristic for finite order case}

\begin{prop} \label{prop:finiteorder}
Let $g$ be an automorphism of finite order of a K3 surface $X$ over $K$ in characteristic $0$.
If $\spen(g) = 1$, then the order of $g$ is a power of the residue characteristic $p$.
\end{prop}

\begin{proof}
By replacing $g$ with a power, we may assume $g$ is of prime order $l$.

We have $g^* \omega = \zeta \omega$ with $\zeta$ an $l$-th root of $1$,
where $\omega$ is as in Lemma \ref{lem:Omega2-RDPmodel}.
Since $\spen(g) = 1$, we have $\norm{\zeta - 1}_p < 1$. 
If $g$ is non-symplectic ($\zeta \neq 1$), this implies $l = p$. 

Assume now $g$ is symplectic.
Any symplectic automorphism on a K3 surface of finite prime-to-characteristic order has at least one fixed point (Lemma \ref{lem:sympfin}), 
so take $x \in \Fix(g)$. We may assume $x$ is $K$-rational.
Take a proper RDP scheme $g$-model $\cX$ (use Proposition \ref{prop:G-projectivemodel} to find such $\cX$)
and let $x_0 \in \cX_0$ be the specialization of $x$.
We can diagonalize the action of $g$ on $\cO_{\cX,x_0}$ as $(x_1, \dots, x_n) \mapsto (a_1 x_1, \dots, a_n x_n)$ ($n = 2$ or $n = 3$)
where $a_i$ are $l$-th roots of $1$ .
Since this action is nontrivial, at least one of $a_i$ is nontrivial, and if $l \neq p$ then its reduction to $\cX_0$ is still nontrivial.
\end{proof}

\begin{cor}
If $p \geq 23$, then no nontrivial automorphism of finite order of a K3 surface over $K$  specializes to the identity.
\end{cor}

\begin{proof}
A K3 surface in characteristic $0$ does not admit an automorphism of prime order $\geq 23$
(\cite{Nikulin:auto}*{Sections 3,5}).
\end{proof}

\begin{rem}
The converse of Proposition \ref{prop:finiteorder} does not hold in general, 
that is, there exists automorphisms of order $p$ specializing to a nontrivial automorphism,
as will be seen for the case $p = 11$ in Example \ref{ex:11}. 
However, if $p \in \{13,17,19 \}$, then the converse is true,
as there is only one K3 surface with automorphism of order $p$,
and in that case the automorphism specializes to identity, as we see in Section \ref{subsec:131719}.
\end{rem}

In the next two subsections we give examples of a K3 surface over $\bQ_p(\zeta_p)$
equipped with a non-symplectic (resp.\ symplectic) automorphism of order $p$ ($2 \leq p \leq 19$ (resp.\ $2 \leq p \leq 7$))
which specializes to identity.
The strategy of the construction is simple: 
We give (an open subscheme of) a proper RDP model on which the automorphism $g$ acts as $g \colon (x_i) \mapsto (a_i x_i)$ with some $p$-th roots $a_i$ of $1$. 
Since $p$-th roots of $1$ are congruent to $1$ modulo the maximal ideal of $\bZ_p[\zeta_p]$, $\spen(g)$ is clearly trivial.
We only need to check that the model is indeed an RDP model (i.e.\ that there are no worse singularities) 
and that $g$ is not trivial on the generic fiber.

\subsection{Non-symplectic examples of finite order} \label{subsec:ker-p-ns}

For $3 \leq p \leq 19$,
let $X_p$ the example of \cite{Kondo:trivially}*{Section 7} of a K3 surface in characteristic $0$ with a non-symplectic automorphism $\sigma$ of order $p$.
Explicitly, $X_{p}$ and $\sigma = \sigma_p$ is given by the Weierstrass form
\begin{align*}
 X_{ 3} &\colon y^2 = x^3 - t^5 (t-1)^5 (t+1)^2, & \sigma(x,y,t) &= (\zeta_{ 3} x, y, t), \\
 X_{ 5} &\colon y^2 = x^3 + t^3 x + t^7, & \sigma(x,y,t) &= (\zeta_{ 5}^3 x, \zeta_{ 5}^2 y, \zeta_{ 5}^2 t), \\
 X_{ 7} &\colon y^2 = x^3 + t^3 x + t^8, & \sigma(x,y,t) &= (\zeta_{ 7}^3 x, \zeta_{ 7}   y, \zeta_{ 7}^2 t), \\
 X_{11} &\colon y^2 = x^3 + t^5 x + t^2, & \sigma(x,y,t) &= (\zeta_{11}^5 x, \zeta_{11}^2 y, \zeta_{11}^2 t), \\
 X_{13} &\colon y^2 = x^3 + t^5 x + t,   & \sigma(x,y,t) &= (\zeta_{13}^5 x, \zeta_{13}   y, \zeta_{13}^2 t), \\
 X_{17} &\colon y^2 = x^3 + t^7 x + t^2, & \sigma(x,y,t) &= (\zeta_{17}^7 x, \zeta_{17}^2 y, \zeta_{17}^2 t), \\
 X_{19} &\colon y^2 = x^3 + t^7 x + t,   & \sigma(x,y,t) &= (\zeta_{19}^7 x, \zeta_{19}   y, \zeta_{19}^2 t), 
\end{align*}
where $\zeta_{p}$ is a primitive $p$-th root of unity.
Non-symplecticness can be checked by computing the action on a global $2$-form $\omega = y^{-1} dx \wedge dt$.

\begin{prop}
Let $2 \leq p \leq 19$ be a prime.
Let $X$ be either $X_{p,p}$ in Section \ref{sec:proof} ($2 \leq p \leq 11$) or $X_p$ above ($3 \leq p \leq 19$) over $K = \bQ_p(\zeta_p)$,
and $\sigma$ the corresponding automorphism of order $p$.
Then $\tilde X$ has potential good reduction, 
and we have $\spen(\sigma) = \id$.
Hence $\sigma \in \Aut(\tilde X)$ is not extendable.
\end{prop}

\begin{proof}
We will see that the same equation defines an RDP model of $X$. 
Then by Lemma \ref{lem:resolve rdp} 
that RDP model admits a simultaneous resolution, 
and then since $\zeta_p = 1$ in $\overline \bF_p$ we have $\spen(\sigma) = \id$,
and $\sigma$ is not extendable by Proposition \ref{cor:criterion for non-ext}(\ref{item:specialization is trivial}).
Since we have already checked $X_{p,p}$ in Section \ref{sec:proof},
it remains to check $X_p$ is an RDP model.

On both fiber of $X_3$,
there are two $E_8$ at $(x,y,t) = (0,0,0),(0,0,1)$
and one $A_2$ at $(0,0,-1)$.
The generic fiber has no other singularities.
The special fiber has one more $A_2$ at $(x',y',s) = (1,0,0)$ and no other singularities.

For $5 \leq p \leq 19$, 
the singularities of fibers of $X_p$ are as follows,
where $c_p = - 4/27$ if $p = 5,7$ and $c_p = - 27/4$ if $p = 11,13,17,19$ 
and $b_p = (-3/2)(a_6/a_4)$, 
where $a_{2i}$ is the coefficient of $x^{3-i}$.

\begin{tabular}{llllllll}
	\toprule
$p$ & & 5 & 7 & 11 & 13 & 17 & 19 \\ 
\midrule
$(x,y,t) = (0,0,0)$ & (both fibers) & 
$E_7$  &
$E_7$  &
$A_2$  &
--- &
$A_2$  &
---  \\
$(x',y',s) = (0,0,0)$ & (both fibers) & 
$E_8$ & 
$E_6$ & 
$E_7$ & 
$E_7$ & 
$A_1$ & 
$A_1$  \\
$(x,y,t) = (b_p, 0, c_p^{1/p})$ & (special fiber) & 
$A_4$    &
$A_6$    &
$A_{10}$ &
$A_{12}$ &
$A_{16}$ &
$A_{18}$  \\
\bottomrule
\end{tabular}

\end{proof}
\begin{rem}
Actually, 
the automorphism $\sigma$ induces a $\mu_p$-action on the special fiber of $X_p$.
Such actions will be studied in a subsequent paper \cite{Matsumoto:k3mun}.
\end{rem}
\begin{rem}
For $p \in \{13,17,19\}$, $\spen(\sigma_p) = \id$ also follows from Dolgachev--Keum's result \cite{Dolgachev--Keum:auto}*{Theorem 2.1}
that K3 surfaces in characteristic $p$ do not admit automorphisms of order $p$ if $p \geq 13$.

For $p \geq 5$, 
potential good reduction of $X_p$ can be shown by the following argument.
Since $\sigma$ is a non-symplectic automorphism 
the field $\bQ(\zeta_p)$ acts on $T(X_p)_\bQ$, where $T$ denotes the transcendental lattice and $_\bQ$ denotes $\otimes \bQ$.
By using the formula 
\[
\rho \geq 2 + \sum_{F \colon \text{fiber}} ((\text{the number of irreducible components in $F$}) -1) ,
\]
where $\sum$ is taken over (non-smooth) fibers $F$ of $X_p \to \bP^1$,
we can easily check that $\rank_{\bQ(\zeta_p)} T(X_p)_\bQ = 1$, i.e.\ $X_p$ has complex multiplication by $\bQ(\zeta_p)$.
Then by \cite{Matsumoto:goodreductionK3}*{Theorem 6.3} $X_p$ has potential good reduction.
(The cited theorem has an assumption on the residue characteristic, 
but under the presence of elliptic fibration it can be weakened to $p \geq 5$ using argument for case (c) after Lemma 3.1 of \cite{Matsumoto:goodreductionK3}.)
\end{rem}

\subsection{Non-symplectic automorphisms of order $13,17,19$} \label{subsec:131719}

\begin{prop} \label{prop:131719}
Let $l \in \{13,17,19\}$.

(1) There exists (up to isomorphism) a unique K3 surface in characteristic $0$ equipped with an automorphism group of order $l$,
and is isomorphic to $(X_l, \spanned{\sigma})$ defined in Section \ref{subsec:ker-p-ns}.

(2) $X_l$ has potential good reduction over $\bQ_p$ for any $p$ including $l$,
and $\sigma$ is extendable if and only if $p \neq l$.
\end{prop}

\begin{proof}
(1) 
This is (announced in \cite{Vorontsov:automorphisms}*{Theorem 7} and)
proved by Oguiso--Zhang \cite{Oguiso--Zhang:non-symplectic}*{Corollary 3}. 

(2) 
The case $p = l$ is done in the previous proposition. Assume $p \neq l$.

If $p \neq 2$ (and $p \neq l$), 
we easily observe that
the singularity of $X_l$ in characteristic $p$ 
is the same to that in characteristic $0$.
If $p = 2$ and $l = 17$, 
we use another coordinate 
$x_1 = 2^{-14/17} x$, $y_1 = 2^{-21/17}(y+t)$, $t_1 = 2^{-4/17} t$.
Then the equation is $- y_1(y_1 - t_1) + x_1^3 + t_1^7 x_1 = 0$,
and the singularity in characteristic $2$ is the same to that in characteristic $0$
(an $A_2$ at $(x_1,y_1,t_1) = (0,0,0)$
and an $A_1$ at $(x_1',y_1',s_1') = (0,0,0)$).
In both cases, we have a canonical simultaneous resolution as in the first part of the proof of Lemma \ref{lem:resolve rdp},
and $\sigma$ extends to that proper smooth model.

If $p = 2$ and $l = 13$ (resp.\ $l = 19$),
in addition to the RDP $(x',y',s) = (0,0,0)$ of the same type $E_7$ (resp.\ $A_1$) to that in characteristic $0$, 
there are extra singularities in characteristic $2$:
$(x,y,t) = (a^5,a,a^2)$ (resp.\ $(a^7,a,a^2)$) are RDPs of type $A_1$
for the $13$-th (resp.\ $19$-th) roots $a$ of $1$,
and $\sigma$ acts on these points cyclically.
The stabilizer of each point is trivial, in particular symplectic.
First we resolve $(x',y',s) = (0,0,0)$ as in the previous case,
and then apply Theorem \ref{thm:criterion}(\ref{thm:criterion:symplectic}) to obtain a proper smooth $\sigma$-model.
\end{proof}

\begin{example} \label{ex:11}
For $l \leq 11$ the situation is different.
The following is a $1$-dimensional example over $K$ of residue characteristic $11$
in which extendability depends on the parameter.

For each $q \in K$, consider the RDP K3 surface and the (non-symplectic) automorphism defined by the equation
\[
 y^2 = x^3 + x + (t^{11} - q)
\]
and $g \colon (x,y,t) \mapsto (x, y, \zeta t)$,
$\zeta = \zeta_{11}$.
This is one of the two $1$-dimensional families 
in the classification of Oguiso--Zhang \cite{Oguiso--Zhang:order-11} 
of K3 surfaces equipped with automorphisms of order $11$.

Letting $b = \sqrt{-1/3}$,
$r = (q + 2b^3)^{1/11}$,
$x' = x - b$, $w = t - r$,
and $a_i = (\zeta^i - 1) / (\zeta - 1)$,
we have
\[
 y^2 = x'^3 + 3b x'^2 + \prod_{i = 0}^{10} (w - a_i r (\zeta-1)), 
\]
$g \colon (x',y,w) \to (x',y, \zeta w + r (\zeta - 1))$.

If $\norm{q^2 + 4/27} < \norm{11}^{-22/10}$,
equivalently $\norm{r (\zeta - 1)} < 1$, 
(where $\norm{\cdot} = \norm{\cdot}_{11}$ is the $11$-adic norm,)
then this equation defines a proper RDP model and 
we have $\spen(g) = \id$, hence $g$ is not extendable.

If $\norm{q^2 + 4/27} \geq \norm{11}^{-22/10}$,
equivalently $\norm{r (\zeta - 1)} \geq 1$, 
then letting $\alpha = ((r(\zeta - 1))^{11})^{-1/6}$, 
$X = \alpha^2 x'$, $Y = \alpha^3 y$, $u = w / (r(\zeta - 1))$,
we have a proper smooth model
\[
 Y^2 = X^3 + 3b \alpha^2 X^2 + \prod (u - a_i),
\]
$g \colon (X,Y,u) \mapsto (X, Y, \zeta u + 1)$.
Thus $g$ is extendable.

(Dolgachev--Keum \cite{Dolgachev--Keum:order11} gave a classification of 
a K3 surface in characteristic $11$ equipped with an automorphism of order $11$: 
it is either of the form 
\[
 X_\varepsilon \colon y^2 + x^3 + \varepsilon x^2 + (u^{11} - u) = 0,
\quad (x,y,u) \mapsto (x,y,u+1),
\]
which is the case in this example, 
or a nontrivial torsor (of order $11$) of such an elliptic surface.)
\end{example}

\subsection{Symplectic examples of finite order} \label{subsec:ker-p-s}

In this section we give, for each prime $2 \leq p \leq 7$, 
an example of a K3 surface $X = X_p$ defined over $K = \bQ_p(\zeta_p)$
and equipped with a symplectic automorphism $\sigma$ of order $p$
which specializes to identity.
Moreover our $X_p$ admits a projective smooth model (over some finite extension) for $p = 5,7$.

Again, these examples may be considered as $\mu_p$-actions on RDP K3 surfaces in characteristic $p$
(see \cite{Matsumoto:k3mun}).

We denote by $\mu_m$ the group of $m$-th roots of $1$ 
and $\zeta_m$ a primitive $m$-th root of $1$ (in the algebraic closure of a field of characteristic $0$).

Case $p = 7$.
Let $X$ be the double sextic K3 surface defined by 
\[
 w^2 + x_1^5 x_{2} + x_{2}^5 x_{3} + x_{3}^5 x_1 = 0.
\]
We have $f \colon \mu_{126}/\mu_3 \injto \Aut(X)$ by $f(t) \colon (w,x_i) \mapsto (w, t^{(-5)^i} x_i)$ for $t \in \mu_{126}$.
Since $f(t)^*$ acts on $H^0(X, \Omega^2_X)$ by $t^{21}$,
we have $f \colon \mu_{21}/\mu_3 \injto \Autsymp(X)$,
where $\Autsymp$ is the group of symplectic automorphisms.
The existence of a symplectic automorphism of order $7$ implies $\rho \geq 19$ (Corollary \ref{cor:Picardnumber}) 
where $\rho$ is the geometric Picard number of $X$.
The existence of an automorphism acting on $H^0(\Omega^2_X)$ by order $3$ 
implies $2 \divides (22 - \rho)$ (since $\bQ(\mu_3)$ acts on $T(X) \otimes \bQ$). 
Hence $\rho = 20$.
It is proved in \cite{Matsumoto:SIP}*{Corollary 0.5} that a K3 surface with $\rho = 20$
admits a projective smooth model after extending $K$ if $p \geq 5$
(projectivity is not explicitly mentioned but follows from the proof).

We observe that the above equation defines a proper RDP model of $X$ 
(the special fiber has 3 RDPs of type $A_6$ at $(w,x_1,x_{2},x_{3}) = (0,1,1,4)$, $(0,1,4,1)$, $(0,4,1,1)$).
So we can compute $\spen(f(\zeta_7))$ using this model, and it is trivial.

Case $p = 5$.
Let $X$ be the quartic K3 surface defined by 
\[
 x_1^3 x_{2} + x_{2}^3 x_{3} + x_{3}^3 x_{4} + x_{4}^3 x_1 = 0.
\]
We have $f \colon \mu_{80}/\mu_4 \injto \Aut(X)$ by 
$f(t) \colon (x_i) \mapsto (t^{(-3)^{i}} x_i)$ for $t \in \mu_{80}$.
Since $f(t)^*$ acts on $H^0(X, \Omega^2_X)$ by $t^{-20}$,
we have $f \colon \mu_{20}/\mu_4 \injto \Autsymp(X)$.
The above equation again defines a proper RDP model 
(the special fiber has 4 RDPs of type $A_4$ at 
$(x_1,x_{2},x_3,x_{4}) = (1,-2a^3,2a^2,a)$ for each primitive $8$-th root $a$ of $1$).

It remains to show $\rho = 20$.
We have another symplectic automorphism $\tau \colon (x_i) \to (\zeta_{40}^{i} x_{i+1})$.
Applying Corollary \ref{cor:Picardnumber} to the group generated by $f(\mu_{20}/\mu_4)$ and $\tau$
(which has $1,5,10,4$ elements of order $1,2,4,5$ respectively) we obtain $\rho \geq 19$.
The existence of an automorphism acting on $H^0(\Omega^2_X)$ by order $4$ 
(e.g.\ $f(\zeta_{80})$)
implies $2 \divides (22 - \rho)$ (since $\bQ(\mu_4)$ acts on $T(X) \otimes \bQ$). 
Another proof of $\rho = 20$
is by finding $20$ independent lines among the $52$ lines given in Section \ref{sec:3}.

Case $p = 3$.
Let $X$ be the double sextic K3 surface over $K$ defined by 
\[
 w^2 + x_0^6 + x_1^6 + x_2^6 + x_0^2 x_1^2 x_2^2 = 0.
\]
Define $g \in \Autsymp(X)$ by $g \colon (w, x_0, x_1, x_2) \mapsto (w, x_0, \zeta_{3} x_1, \zeta_{3}^2 x_2)$.
The above equation defines a proper RDP model
(the special fiber has 6 RDPs of type $A_2$ at 
$(w = x_0 x_1 x_2 = x_0^2 + x_1^2 + x_2^2 = 0)$).

Case $p = 2$.
Let $X$ be the quartic K3 surface over $K$ defined by 
\[
 w^3 x + w x^3 + y^3 z + y z^3 + wxyz = 0.
\]
Define $g \in \Autsymp(X)$ by $g \colon (w,x,y,z) \mapsto (w,x,-y,-z)$.
The above equation defines a proper RDP model
(the special fiber has 4 RDPs of type $A_3$ at 
$(w,x,y,z) = (0,1,1,1),(1,0,1,1),(1,1,0,1),(1,1,1,0)$).

\subsection{Examples of infinite order} \label{subsec:ker-infinite}

In this section we give examples, in all residue characteristic $p \geq 2$,
of automorphisms of infinite orders that specializes to the identity,

Consider a K3 surface $X$ equipped with an elliptic fibration $X \to \bP^1$,
and a non-torsion section $Z \subset X$ of the fibration.
Assume $X$ admits a projective RDP model with an elliptic fibration $\cX \to \bP^1_{\cO_K}$
and that the specialization of $Z$ is the zero section plus some fibral components.
Then the translation $\phi \colon X \to X$ by $Z$ specializes to the identity on $\cX_0$.
It is known that translation on an elliptic K3 surface is symplectic \cite{Huybrechts:lecturesK3}*{Lemma 16.4.4}.

Now we give an explicit example.
Let $X$ be the elliptic K3 surface defined by the equation $-y^2 - xy + x^3 - p^{12} x + t^6(t^6+p^6) = 0$.
Let $Z$ be the section defined by $(x, y) = (t^6(t^6+p^6)p^{-12}, t^{12}(t^6+p^6)p^{-18})$.
The singularity of the special fiber of $X$ is as follows.
An $A_{11}$ at $(x = y = t = 0)$ for any $p$.
If $p = 3$, an $E_6$ at $(x',y',s) = (-1,0,0)$. 
If $p = 2$, an $D_7$ at $(x',y',s) = (0,-1,0)$.

$\phi$ has infinite order since its restriction to the fiber $(t = 1)$, which is a smooth elliptic curve over $\bQ$, has infinite order
by a Lutz-Nagell type result (\cite{Silverman:AEC}*{Theorem VII.3.4}).
Then,  for any $m \geq 1$, $\phi^m$ is not extendable
since $\phi^m \neq \id$ and $\spen(\phi^m) = \id$.

Next let $\sigma$ be the automorphism $(x,y,t) \mapsto (x,y, \zeta_6 t)$.
Then the composite $\phi \sigma$ is not extendable since its power $(\phi \sigma)^6 = \phi^6$ is not extendable, 
and $\phi \sigma$ is non-symplectic since $\phi$ is symplectic and $\sigma$ is not.

Similar example would exist also in equal characteristic $0$.
Also, Oguiso \cite{Oguiso:localfamilies}*{Theorem 1.5(2)} gave an example of $1$-dimensional family $\{ X_t \}_{t \in \Delta}$ of complex K3 surfaces 
with $\Aut(X_t)$ are infinite for $t$ outside a countable subset of $\Delta$, but $\Aut(X_0)$ is finite.

\section{An example in characteristic $3$} \label{sec:3}

In this section we give an example of a K3 surface $X_K$ over $K = \bQ_{3^4} = \bQ_3(\zeta_{80})$ equipped with 
an automorphism $g_K$ defined over $K$
such that
the characteristic polynomial of $\spen(g_K)$ is irreducible.
By 
Corollary \ref{cor:criterion for non-ext}(\ref{item:specialization is irreducible})
this gives another example 
of Theorem \ref{thm:non-extendable} for $G = \bZ$, $p = 3$.
Apart from the non-extendability, the existence of $g_K$ with the characteristic polynomial of $\spen(g_K)^*$ being irreducible 
would be itself interesting.
The proof of irreducibility, however, requires hard computations.

Let $X_k$ be the Fermat quartic $(F = w^4 + x^4 + y^4 + z^4 = 0)$ in $\bP^3_k$ over $k = \bF_{3^4}$.
(This is the (unique) supersingular K3 surface with Artin invariant $1$ in characteristic $3$, but we do not need this fact.)
Kondo--Shimada determined the lines on $X_k$ and their explicit equations
and showed that $\NS(X_{\overline k}) = \NS(X_k)$ is generated by those lines.
We use their notation $l_1, \ldots, l_{112}$ of \cite{Kondo--Shimada:3}\footnote{
Table 2 in the published version has errors (e.g.\ the formulas for $l_3$ and $l_5$ are the same).
Instead we refer to Table 3.1 in arXiv version ({\tt arXiv:1205.6520v2}).}.
 
Another coordinate 
$(u_1, u_4, u_2, u_3) = 
 (w, x, y, z) M^{-1}$,
 where $M$ is the matrix
\[
M = \begin{pmatrix}
\zeta^2-\zeta^3 & -1-\zeta^2 & -1+\zeta-\zeta^2 & \zeta-\zeta^4 \\
-\zeta^2+\zeta^3 & -1-\zeta^3 & -1-\zeta^3+\zeta^4 & -\zeta+\zeta^4 \\
\zeta^2-\zeta^4 & \zeta+\zeta^2 & -\zeta^2-\zeta^3+\zeta^4 & -1+\zeta+\zeta^3 \\
-\zeta+\zeta^3 & \zeta^3+\zeta^4 & \zeta-\zeta^2-\zeta^3 & 1-\zeta-\zeta^3 \\
\end{pmatrix},
\]
gives the equation $u_1^3 u_2 + u_2^3 u_4 + u_4^3 u_3 + u_3^3 u_1 = 0$.
Here $\zeta = \zeta_{5} \in \bF_{3^4}$ is a primitive $5$-th root of $1$ satisfying $i = -1 + \zeta + \zeta^{-1}$.
Let $X_K$ be the quartic K3 surface over $K = \bQ_{3^4}$ defined by this equation.

There are the following 52 lines $l^1_{(d,e)}$, $l^2_a$, $l^3$, $l^4$ on $X_{\overline K}$, all defined over $K = \bQ_{3^4}$:
\[l^1_{(d,e)} \colon u_1 + ed u_2 + d^3 u_3 =  u_4 - e^3 d^3 u_2 - d u_3 = 0 \] 
for each of the $40$ solutions $(d,e)$ of $e^5 = 1$ and $d^8 - 3 e^3 d^4 + e = 0$, 
\[ l^2_a \colon u_1 - a u_4 = u_2 + a^7 u_3 = 0 \]
for each of the $10$ solutions $a$ of $a^{10} = 1$,
and $l^3 \colon u_2 = u_3 = 0$ and $l^4 \colon u_1 = u_4 = 0$. 
We observe that there are no more.
We can calculate their specialization to $X_k$. 
For example, the line $u_1 - d'^9 u_2 + d'^3 u_3 =  u_4 + d'^{27} u_2 - d' u_3 = 0 $ on $X_k$,
where $d' \in k$ is an $80$-th root of $1$, 
is the specialization of some $l^1_{(d,e)}$ if and only if $d'^{40} = -1$.
By explicit calculation (omitted) we observe that $l_i$ comes from a line on $X_K$ if and only if $i \in I$, where 
\begin{gather*}
 I = \{ 
1, 2, 3, 4, 5, 9, 10, 13, 15, 18, 
20, 21, 22, 23, 24, 25, 26, 30, 33, 36, \\
37, 40, 41, 44, 45, 48, 51, 52, 57, 63, 
65, 66, 67, 68, 70, 72, 74, 75, 78, 82, \\
86, 93, 98, 101, 102, 103, 104, 106, 109, 110, 
111, 112
 \}.
\end{gather*}

Define divisor classes $D_1$ and $D_2$ on $X_k$ by
\begin{align*}
D_1 &= 3h - (l_{21} + l_{22} + l_{63} + l_{65} + l_{50} + l_{88}), \\
D_2 &= 2h - (l_{65} + l_{66} + l_{70}),
\end{align*}
where $h$ denotes the hyperplane class (with respect to the embedding in $\bP^3$).
Since $l_{50} + l_{88} = h - l_{5} - l_{112}$ (since the hyperplane section $(w + (-1-i)x + i y + (1-i)z = 0)$ is equal to the sum of these $4$ lines),
the classes $D_i$ come from the classes $D_{i,K}$ of $X_K$.

We note that $D_1$ is the class $m_1$ in \cite{Kondo--Shimada:3}.

We easily verify that $D_i$ are nef and that $D_i^2 = 2$, 
and hence $D_{i,K}$ have the same property.
Hence we obtain generically $2$-to-$1$ morphisms $\pi_i \colon X_k \to \bP^2_k$ and $\pi_{i,K} \colon X_K \to \bP^2_K$.
\begin{claim} \label{claim:exceptional}
(1)
The exceptional divisors of $\pi_1$ are
\[
(l_{10}, l_{18}), 
(l_{16}, l_{99}), 
(l_{29}, l_{49}), 
(l_{60}, l_{73}), 
(l_{23}), 
(l_{37}), 
(l_{62}), 
(l_{68}), 
(l_{102}), 
(l_{112}), 
\]
and those of $\pi_2$ are
\[
(l_{67}, l_{68}), 
(l_{90}, l_{94}), 
(l_{49}), 
(l_{54}), 
(l_{60}), 
(l_{63}), 
(l_{69}), 
(l_{97}), 
(l_{102}), 
(l_{107}), 
(l_{112}), 
\]
where the parentheses denote connected components.

(2)
The exceptional divisors of $\pi_{1,K}$ are
\[
(\tilde l_{10}, \tilde l_{18}), 
(C_{16,99}), 
(\tilde l_{23}), 
(\tilde l_{37}), 
(\tilde l_{68}), 
(\tilde l_{102}), 
(\tilde l_{112}), 
\]
and those of $\pi_{2,K}$ are
\[
(\tilde l_{67}, \tilde l_{68}), 
(C_{90,94}), 
(\tilde l_{63}), 
(\tilde l_{102}), 
(\tilde l_{112}), 
\]
where $\tilde l_i$ is the (unique) line on $X_K$ specializing to $l_i$
and $C_{i,j}$ is the (unique) rational curve on $X_K$ specializing to $l_i + l_j$.
\end{claim}
We prove this later (in a brutal way). For $\pi_1$ this is already showed in \cite{Kondo--Shimada:3} but we give another proof. 

Let $\tau_i$ be the involutions on $X_k$ induced by the deck transformations of $\pi_i$.
Note that $\tau_i$ are the specializations of the involutions $\tau_{i,K}$ on $X_K$ defined by the classes $D_{i,K}$.
Using the previous claim
we can compute the $+1$-parts of $\tau_{i,K}^*$ and $\tau_{i}^*$ on $\Het^2$:
the $+1$-part is freely generated by the pull-back of $\cO_{\bP^2}(1)$ 
and the classes of connected components of the exceptional divisor
(provided these components are all $A_1$ or $A_2$).
By Proposition \ref{prop:criterion for non-ext},
$\tau_{i,K}$ are not extendable to proper smooth models.

We need one more automorphism.
Let $\sigma$ and $\sigma_K$ be the diagonal linear transformations 
$(u_1, u_4, u_2, u_3) \mapsto (u_1, - u_4, i u_2, -i u_3)$ on $X_k$ and $X_K$.
(We also have a more symmetric formula $(u_1, u_4, u_2, u_3) \mapsto (\zeta_{16} u_1, \zeta_{16}^{9} u_4, \zeta_{16}^{-3} u_2, \zeta_{16}^{-27} u_3)$,
where $\zeta_{16} = -1 + \zeta + \zeta^3$ is a $4$-th root of $-i$.)
(A linear automorphism diagonalized by this kind of basis also appears in \cite{Kondo--Shimada:3}*{Example 3.4}.)

Now let $g = \sigma \tau_2 \tau_1 \tau_2$.
Clearly $g$ is the specialization of $g_K = \sigma_K \tau_{2,K} \tau_{1,K} \tau_{2,K}$.

\begin{claim}
The characteristic polynomial of $g^*$ on $\Het^2(X_{\overline k}, \bQ_l)$ is equal to $F(x) = $
\begin{gather*}
x^{22} - 4x^{21} + 2x^{20} - 3x^{18} + 4x^{17} - 5x^{16} + x^{15} + x^{14} - 2x^{13} + 2x^{12} - 3x^{11} \\
+ 2x^{10} - 2x^9 + x^8 + x^7 - 5x^6 + 4x^5 - 3x^4 + 2x^2 - 4x + 1
\end{gather*}
and is irreducible.
\end{claim}

\begin{proof}
We first prove irreducibility of this polynomial $F$.
We have several ways.
(1) We can ask a mathematical software (e.g.\ {\tt SageMath}).
(2) The irreducible decompositions of $F \bmod 2$ and $F \bmod 3$ imply irreducibility (we omit the details).
(3) Assuming that $F$ is the characteristic polynomial of $g^*$ on $\Het^2$
(and hence of $g^*$ on $\NS(X_{\overline k})$),
it has at most one non-cyclotomic irreducible factor by the following lemma. 
So it suffices to check $F$ is prime to any cyclotomic polynomial of degree $\leq 22$ (we omit the verification).

\begin{lem}[\cite{McMullen:dynamicsK3}*{Corollary 3.3}] \label{lem:Salem}
Let $f$ be an isometry of a lattice $L$ (over $\bZ$) of signature $(+1, -(r-1))$ 
and assume $f$ preserves a connected component of $\{ x \in L  \mid x^2 > 0 \} $.
Then the characteristic polynomial of $f$ has at most one non-cyclotomic irreducible factor.
Moreover that factor (if exists) is a Salem polynomial, that is, 
an irreducible monic integral polynomial that has exactly two real roots, $\lambda > 1$ and $\lambda^{-1}$, 
and the other roots (if any) lie on the unit circle.
\end{lem}

Since $\Het^2(X_{\overline k}, \bQ_l)$ is generated by algebraic cycles (defined over $k$), 
it suffices to compute the action on $\NS(X_k) \otimes \bQ$.

The transformation matrix of $\tau_1$ with respect to the basis $\beta_1 = $
\begin{gather*}
 \{
l_{23},
l_{37},
l_{62},
l_{68},
l_{102},
l_{112},
l_{10} + l_{18},
l_{16} + l_{99},
l_{29} + l_{49},
l_{60} + l_{73},
D_1, \\
l_{10} - l_{18},
l_{16} - l_{99},
l_{29} - l_{49},
l_{60} - l_{73},
l_{2} - l_{33},
l_{4} - l_{11},
l_{5} - l_{24},
l_{7} - l_{85}, \\
l_{13} - l_{67},
l_{30} - l_{87},
2 l_{3}+l_{112} - (l_{10} + l_{18} + l_{16} + l_{99} + l_{90} + l_{94})
\} 
\end{gather*}
is $T'_1 = \diag(\underbrace{1,\ldots, 1}_{11},\underbrace{-1,\ldots, -1}_{11})$.

The transformation matrix of $\tau_2$ with respect to the basis $\beta_2 = $
\begin{gather*}
 \{
l_{67} + l_{68},
l_{90} + l_{94},
l_{49},
l_{54},
l_{60},
l_{63},
l_{69},
l_{97},
l_{102},
l_{107},
l_{112},
D_2, \\
l_{67} - l_{68},
l_{90} - l_{94},
l_{45} - l_{82},
l_{24} - l_{75},
l_{36} - l_{79},
l_{30} - l_{81}, \\
l_{39} - l_{76},
l_{25} - l_{86},
l_{42} - l_{85},
l_{10} - l_{18}
\}
\end{gather*}
is $T'_2 = \diag(\underbrace{1,\ldots, 1}_{12},\underbrace{-1,\ldots, -1}_{10})$.

The transformation matrix of $\sigma$ with respect to the basis $\beta_3 = $
\[ \{
l_{7}, l_{107}, l_{95}, l_{14}, l_{83}, l_{92}, l_{43}, l_{69}, l_{34}, l_{56}, l_{11}, l_{59}, l_{80}, l_{16}, l_{50}, l_{85}, l_{100}, l_{61}, l_{27}, l_{29}, l_{15}, l_{20}
\} \] 
is the $5$-th power of the matrix
\[
R = \begin{pmatrix}
      &   &   & 1  &  &  \\
1     &   &  &   &  &  \\
      & \ddots  &   &  &  &  \\
      &  & 1 &  &  &  \\
      &  &  &  & 1 &  \\
      &  &  &  &  & 1 \\
\end{pmatrix}.
\]
(More precisely, $\sigma$ is the $5$-th power of the linear automorphism 
$\rho \colon (u_1, u_4, u_2, u_3) \mapsto (\zeta_{80} u_1, \zeta_{80}^{9} u_4, \zeta_{80}^{-3} u_2, \zeta_{80}^{-27} u_3)$,
where $\zeta_{80} =  \zeta - \zeta^3$ satisfies $\zeta_{80}^5 = \zeta_{16}$,
and $\rho$ acts on $\beta_3$ by $R$.)

From these information we can compute the action and the characteristic polynomial.
Define $\psi \colon \NS(X_k) \otimes \bQ \isomto \bQ^{22}$
to be the isomorphism defined by $ \psi(v) = (v \cdot l)_{l \in \beta_3}$.
Let $B_i$ be the matrices consisting of column vectors $\psi(v)$ ($v \in \beta_i$).
Then $T_i = (B_i^{-1} B_3)^{-1} T'_i (B_i^{-1} B_3)$ (for $i = 1,2$) 
are the transformation matrices of $\tau_i$ with respect to the basis $\beta_3$.
It remains to check that the characteristic polynomial of
$ R^5 T_2 T_1 T_2  $
is equal to $F$ (omitted).
We write down the $B_i$ for convenience.
\[ B_1 = \left(
\begin{smallmatrix}
0 & 1 & 1 & 0 & 0 & 0 & 0 & 0 & 0 & 1 & 1 & 0 & 0 & 0 & 1 & 0 & -1 & 1 & -2 & 0 & 1 & 1 \\ 
1 & 0 & 1 & 0 & 0 & 0 & 0 & 1 & 0 & 0 & 1 & 0 & -1 & 0 & 0 & 1 & 0 & 0 & 0 & 0 & 1 & -1 \\ 
0 & 0 & 0 & 0 & 0 & 0 & 0 & 0 & 1 & 1 & 1 & 0 & 0 & -1 & 1 & 1 & 0 & 1 & 0 & -1 & 0 & -1 \\ 
1 & 0 & 1 & 0 & 0 & 1 & 1 & 1 & 1 & 2 & 3 & 1 & 1 & -1 & 0 & 0 & 0 & 0 & 0 & 0 & 1 & -2 \\ 
1 & 0 & 0 & 0 & 1 & 0 & 0 & 0 & 0 & 1 & 1 & 0 & 0 & 0 & 1 & 0 & 0 & 0 & 0 & -1 & -1 & -1 \\ 
0 & 0 & 0 & 1 & 0 & 0 & 0 & 1 & 1 & 1 & 2 & 0 & -1 & -1 & 1 & 0 & 0 & 0 & -1 & 0 & 1 & -2 \\ 
0 & 0 & 0 & 0 & 0 & 0 & 1 & 0 & 0 & 1 & 1 & -1 & 0 & 0 & -1 & 0 & 0 & 0 & 0 & -1 & 0 & -1 \\ 
0 & 0 & 0 & 0 & 0 & 0 & 0 & 1 & 0 & 1 & 1 & 0 & -1 & 0 & -1 & 1 & 0 & 1 & 0 & 0 & -1 & -1 \\ 
0 & 0 & 1 & 0 & 0 & 0 & 1 & 1 & 0 & 1 & 2 & 1 & -1 & 0 & 1 & -1 & 1 & 1 & 0 & -1 & -1 & -2 \\ 
0 & 1 & 0 & 0 & 0 & 1 & 1 & 1 & 1 & 1 & 2 & 1 & -1 & 1 & 1 & 0 & 0 & 0 & 0 & 0 & 0 & -1 \\ 
1 & 0 & 0 & 1 & 0 & 0 & 1 & 1 & 1 & 1 & 2 & 1 & -1 & -1 & -1 & -1 & 2 & 0 & 1 & 0 & 0 & 0 \\ 
0 & 0 & 0 & 1 & 0 & 1 & 0 & 0 & 0 & 1 & 1 & 0 & 0 & 0 & 1 & -1 & 0 & 0 & 0 & 0 & -1 & 0 \\ 
0 & 0 & 1 & 1 & 0 & 0 & 1 & 0 & 0 & 0 & 1 & 1 & 0 & 0 & 0 & 0 & 0 & 0 & 0 & 0 & 0 & -1 \\ 
0 & 0 & 0 & 0 & 0 & 0 & 0 & -1 & 0 & 0 & 0 & 0 & -3 & 0 & 0 & 0 & 1 & 0 & 0 & 1 & 0 & 0 \\ 
0 & 1 & 1 & 1 & 1 & 1 & 1 & 1 & 1 & 1 & 3 & 1 & -1 & -1 & -1 & 0 & 1 & 1 & 0 & 0 & 1 & -2 \\ 
0 & 1 & 1 & 0 & 0 & 0 & 0 & 0 & 0 & 1 & 1 & 0 & 0 & 0 & -1 & 0 & 1 & -1 & 2 & 0 & -1 & -1 \\ 
1 & 1 & 1 & 1 & 0 & 1 & 1 & 1 & 1 & 1 & 3 & -1 & -1 & -1 & 1 & 1 & 0 & 0 & 0 & 0 & -1 & -2 \\ 
0 & 0 & 1 & 0 & 0 & 1 & 0 & 0 & 1 & 0 & 1 & 0 & 0 & -1 & 0 & -1 & 1 & 0 & 0 & -1 & 0 & 0 \\ 
1 & 1 & 0 & 0 & 0 & 0 & 0 & 0 & 1 & 0 & 1 & 0 & 0 & -1 & 0 & 0 & 0 & 0 & 0 & 0 & 0 & -1 \\ 
0 & 0 & 0 & 0 & 0 & 0 & 0 & 0 & -1 & 0 & 0 & 0 & 0 & -3 & 0 & 0 & 1 & 0 & 0 & 0 & 1 & 0 \\ 
0 & 0 & 0 & 0 & 1 & 0 & 1 & 1 & 1 & 0 & 2 & -1 & 1 & 1 & 0 & 0 & -1 & -1 & 0 & 1 & 0 & 0 \\ 
0 & 0 & 1 & 0 & 1 & 0 & 1 & 1 & 1 & 1 & 2 & -1 & -1 & -1 & 1 & 0 & 0 & -1 & 0 & 1 & 0 & 0 \\ 
\end{smallmatrix}
\right), \]
\[ B_2 = \left(
\begin{smallmatrix}
0 & 1 & 0 & 0 & 1 & 0 & 0 & 0 & 0 & 0 & 0 & 1 & 0 & 1 & 0 & -1 & 0 & 1 & 0 & -1 & 0 & 0 \\ 
0 & 0 & 0 & 0 & 0 & 0 & 0 & 0 & 0 & -2 & 0 & 0 & 0 & 0 & 0 & 0 & 0 & 0 & 0 & 0 & 0 & 0 \\ 
1 & 1 & 1 & 1 & 1 & 1 & 0 & 0 & 0 & 0 & 0 & 2 & 1 & -1 & 1 & -1 & 1 & -1 & 1 & -1 & 1 & 0 \\ 
1 & 1 & 1 & 0 & 1 & 0 & 0 & 1 & 0 & 0 & 1 & 2 & 1 & -1 & 0 & 0 & -1 & 1 & 1 & -1 & 0 & 1 \\ 
1 & 1 & 0 & 0 & 1 & 1 & 0 & 1 & 1 & 0 & 0 & 2 & 1 & 1 & -1 & -1 & -1 & -1 & 0 & 0 & 0 & 0 \\ 
1 & 1 & 1 & 1 & 1 & 1 & 0 & 0 & 0 & 0 & 0 & 2 & -1 & 1 & -1 & 1 & -1 & 1 & -1 & 1 & -1 & 0 \\ 
1 & 0 & 0 & 1 & 0 & 0 & 0 & 0 & 0 & 0 & 0 & 1 & 1 & 0 & -1 & 0 & -1 & 0 & 1 & 0 & 1 & -1 \\ 
0 & 0 & 0 & 0 & 0 & 0 & -2 & 0 & 0 & 0 & 0 & 0 & 0 & 0 & 0 & 0 & 0 & 0 & 0 & 0 & 0 & 0 \\ 
1 & 0 & 0 & 0 & 1 & 0 & 0 & 0 & 0 & 0 & 0 & 1 & 1 & 0 & -1 & 0 & 1 & 0 & -1 & 0 & 1 & 1 \\ 
0 & 0 & 0 & 1 & 1 & 0 & 0 & 0 & 0 & 0 & 1 & 1 & 0 & 0 & 0 & 0 & -1 & -1 & -1 & -1 & 1 & 1 \\ 
1 & 0 & 1 & 0 & 0 & 0 & 0 & 0 & 0 & 0 & 0 & 1 & -1 & 0 & -1 & 0 & -1 & 0 & 1 & 0 & 1 & 1 \\ 
1 & 1 & 0 & 0 & 1 & 1 & 0 & 0 & 0 & 1 & 1 & 2 & -1 & -1 & 0 & 0 & 0 & 0 & 1 & 1 & 1 & 0 \\ 
1 & 0 & 0 & 0 & 0 & 0 & 0 & 1 & 0 & 0 & 0 & 1 & -1 & 0 & 1 & 0 & -1 & 0 & -1 & 0 & 0 & 1 \\ 
0 & 1 & 0 & 1 & 0 & 0 & 0 & 0 & 0 & 0 & 0 & 1 & 0 & 1 & 0 & -1 & 0 & -1 & 0 & 1 & 0 & 0 \\ 
1 & 1 & 1 & 1 & 0 & 0 & 0 & 0 & 1 & 0 & 1 & 2 & -1 & -1 & 0 & 0 & 1 & 1 & 0 & 0 & 1 & 1 \\ 
0 & 1 & 0 & 0 & 0 & 0 & 0 & 1 & 0 & 0 & 0 & 1 & 0 & -1 & 0 & 1 & 0 & -1 & 0 & -1 & 3 & 0 \\ 
1 & 1 & 1 & 0 & 1 & 0 & 0 & 1 & 0 & 0 & 1 & 2 & -1 & 1 & 0 & 0 & 1 & -1 & -1 & 1 & 0 & -1 \\ 
1 & 1 & 1 & 0 & 0 & 1 & 1 & 0 & 0 & 0 & 1 & 2 & 1 & 1 & 0 & 0 & -1 & -1 & -1 & -1 & 1 & 0 \\ 
0 & 1 & 1 & 0 & 0 & 0 & 0 & 0 & 0 & 0 & 0 & 1 & 0 & -1 & 0 & -1 & 0 & -1 & 0 & 1 & 0 & 0 \\ 
0 & 0 & 1 & 0 & 0 & 1 & 0 & 1 & 0 & 0 & 0 & 1 & 0 & 0 & -1 & -1 & 1 & 1 & 1 & 1 & 0 & 0 \\ 
0 & 0 & 0 & 0 & 0 & 0 & 1 & 0 & 1 & 1 & 0 & 1 & 0 & 0 & 1 & 1 & 0 & 0 & 0 & 0 & -1 & -1 \\ 
0 & 0 & 1 & 0 & 1 & 0 & 0 & 0 & 1 & 0 & 0 & 1 & 0 & 0 & 1 & 1 & -1 & -1 & 0 & 0 & 1 & -1 \\ 
\end{smallmatrix}
\right), \]
\[B_3 = \left(
\begin{smallmatrix}
-2 & 0 & 0 & 0 & 0 & 0 & 0 & 0 & 0 & 0 & 1 & 0 & 0 & 0 & 0 & 0 & 0 & 0 & 0 & 0 & 1 & 0 \\ 
0 & -2 & 0 & 0 & 0 & 0 & 0 & 0 & 0 & 0 & 0 & 1 & 0 & 0 & 0 & 0 & 0 & 0 & 0 & 0 & 1 & 0 \\ 
0 & 0 & -2 & 0 & 0 & 0 & 0 & 0 & 0 & 0 & 0 & 0 & 1 & 0 & 0 & 0 & 0 & 0 & 0 & 0 & 1 & 0 \\ 
0 & 0 & 0 & -2 & 0 & 0 & 0 & 0 & 0 & 0 & 0 & 0 & 0 & 1 & 0 & 0 & 0 & 0 & 0 & 0 & 1 & 0 \\ 
0 & 0 & 0 & 0 & -2 & 0 & 0 & 0 & 0 & 0 & 0 & 0 & 0 & 0 & 1 & 0 & 0 & 0 & 0 & 0 & 1 & 0 \\ 
0 & 0 & 0 & 0 & 0 & -2 & 0 & 0 & 0 & 0 & 0 & 0 & 0 & 0 & 0 & 1 & 0 & 0 & 0 & 0 & 1 & 0 \\ 
0 & 0 & 0 & 0 & 0 & 0 & -2 & 0 & 0 & 0 & 0 & 0 & 0 & 0 & 0 & 0 & 1 & 0 & 0 & 0 & 1 & 0 \\ 
0 & 0 & 0 & 0 & 0 & 0 & 0 & -2 & 0 & 0 & 0 & 0 & 0 & 0 & 0 & 0 & 0 & 1 & 0 & 0 & 1 & 0 \\ 
0 & 0 & 0 & 0 & 0 & 0 & 0 & 0 & -2 & 0 & 0 & 0 & 0 & 0 & 0 & 0 & 0 & 0 & 1 & 0 & 1 & 0 \\ 
0 & 0 & 0 & 0 & 0 & 0 & 0 & 0 & 0 & -2 & 0 & 0 & 0 & 0 & 0 & 0 & 0 & 0 & 0 & 1 & 1 & 0 \\ 
1 & 0 & 0 & 0 & 0 & 0 & 0 & 0 & 0 & 0 & -2 & 0 & 0 & 0 & 0 & 0 & 0 & 0 & 0 & 0 & 1 & 0 \\ 
0 & 1 & 0 & 0 & 0 & 0 & 0 & 0 & 0 & 0 & 0 & -2 & 0 & 0 & 0 & 0 & 0 & 0 & 0 & 0 & 1 & 0 \\ 
0 & 0 & 1 & 0 & 0 & 0 & 0 & 0 & 0 & 0 & 0 & 0 & -2 & 0 & 0 & 0 & 0 & 0 & 0 & 0 & 1 & 0 \\ 
0 & 0 & 0 & 1 & 0 & 0 & 0 & 0 & 0 & 0 & 0 & 0 & 0 & -2 & 0 & 0 & 0 & 0 & 0 & 0 & 1 & 0 \\ 
0 & 0 & 0 & 0 & 1 & 0 & 0 & 0 & 0 & 0 & 0 & 0 & 0 & 0 & -2 & 0 & 0 & 0 & 0 & 0 & 1 & 0 \\ 
0 & 0 & 0 & 0 & 0 & 1 & 0 & 0 & 0 & 0 & 0 & 0 & 0 & 0 & 0 & -2 & 0 & 0 & 0 & 0 & 1 & 0 \\ 
0 & 0 & 0 & 0 & 0 & 0 & 1 & 0 & 0 & 0 & 0 & 0 & 0 & 0 & 0 & 0 & -2 & 0 & 0 & 0 & 1 & 0 \\ 
0 & 0 & 0 & 0 & 0 & 0 & 0 & 1 & 0 & 0 & 0 & 0 & 0 & 0 & 0 & 0 & 0 & -2 & 0 & 0 & 1 & 0 \\ 
0 & 0 & 0 & 0 & 0 & 0 & 0 & 0 & 1 & 0 & 0 & 0 & 0 & 0 & 0 & 0 & 0 & 0 & -2 & 0 & 1 & 0 \\ 
0 & 0 & 0 & 0 & 0 & 0 & 0 & 0 & 0 & 1 & 0 & 0 & 0 & 0 & 0 & 0 & 0 & 0 & 0 & -2 & 1 & 0 \\ 
1 & 1 & 1 & 1 & 1 & 1 & 1 & 1 & 1 & 1 & 1 & 1 & 1 & 1 & 1 & 1 & 1 & 1 & 1 & 1 & -2 & 0 \\ 
0 & 0 & 0 & 0 & 0 & 0 & 0 & 0 & 0 & 0 & 0 & 0 & 0 & 0 & 0 & 0 & 0 & 0 & 0 & 0 & 0 & -2 \\ 
\end{smallmatrix}
\right). \]
\end{proof}

\begin{proof}[Proof of Claim \ref{claim:exceptional}]
We first prove (2) assuming (1).
Let $C \subset X_K$ be an (irreducible) exceptional curve for $\pi_{i,K}$.
Then the specialization $C_0$ of $C$ to $X_k$ is the sum of exceptional curves and is connected,
hence is either an exceptional curve for $\pi_i$ or the sum of two exceptional curves forming an $A_2$ component. 
Since $C^2 \geq -2$, we observe that all components of $C_0$ have multiplicity $1$.
By checking liftability of the classes, we obtain the stated list.
(The class $l_{16} + l_{99}$ is liftable to a class $C_{16,99}$ of $X_K$ since it is equal to $h - l_{57} - l_{75}$
and the lines $l_{57}$ and $l_{75}$ are liftable. It is irreducible since the lines $l_{16}$ and $l_{99}$ are not liftable.
The class $l_{29} + l_{49}$ is not liftable since it is equal to $h - l_{41} - l_{77}$
and the line $l_{41}$ is liftable and $l_{77}$ is not. The other cases are similar or simpler.)

We now prove (1).
By computing the intersection numbers we see that the above curves are indeed exceptional. 
We need to show there are no more. 
First we consider $\pi_2$. 
We identify $H^0(X_k, \cO(mD_2))$ with the space of homogeneous polynomials of degree $2m$ modulo $F$
with vanishing order at least $m$ at $l_{65}$, $l_{66}$, and $l_{70}$.
Define linear polynomials $f_{65}, g_{65}, f_{70}, g_{70}$ by
\begin{alignat*}{2}
 f_{65} &= w + (1+i)y &&\in H^0(X_k, \cO(h - (l_{65} + l_{66}))), \\ 
 g_{65} &= x + (1+i)z &&\in H^0(X_k, \cO(h - (l_{65}))), \\ 
 f_{70} &= x + (1-i)z &&\in H^0(X_k, \cO(h - (l_{70} + l_{66}))), \\
 g_{70} &= w + (1-i)y &&\in H^0(X_k, \cO(h - (l_{70}))), 
\end{alignat*}
so that they vanish on the indicated lines.
Let $A = f_{65}g_{70}$,
$B = g_{65}f_{70}$,
$C = f_{65}f_{70}$.
Then $A,B,C$ form a basis of $H^0(X_k, \cO(D_2))$.
Let $Y_1 =  (1+i) f_{65}f_{70} (f_{65}^3 g_{70}   + g_{65}^3 f_{70})$
and $Y_2 = (-1+i) f_{65}f_{70} (f_{65}   g_{70}^3 + g_{65}   f_{70}^3)$.
Then we see that $Y_1 - Y_2 = FC \equiv 0 \pmod F$, 
and that $Y_1$ ($= Y_2$) together with the ten cubic monomials of $A,B,C$ form a basis of $H^0(X_k, \cO(3D_2))$.
We obtain the formula $Y_1^2$ ($= Y_1Y_2$) $= A^3B^3 + (A^4+B^4)C^2  + ABC^4$ and
conclude that it has $13$ exceptional curves (forming two $A_2$ and nine $A_1$). 
Hence the list above gives all exceptional curves.

Now we consider $\pi_1$.
We identify $H^0(X_k, \cO(mD_1))$ with the space of homogeneous polynomials of degree $3m$ modulo $F$
with vanishing order at least $m$ at each of $l_{21}$, $l_{22}$, $l_{50}$, $l_{63}$, $l_{65}$, and $l_{88}$.
Define linear polynomials $a,b_1,c_1,d_1,c_2,d_2$ and a quadratic polynomial $\phi_2$ by
\begin{alignat*}{2}
 c_1 &= w + iy + (-i)z & &\in H^0(X_k, \cO(h - (l_{21} + l_{22}))), \\ 
 c_2 &= w + (-i)x + (-1+i)y + (-1-i)z & &\in H^0(X_k, \cO(h - (l_{22} + l_{88}))), \\ 
 d_1 &= w + (1+i)x + (-1-i)y + (-1)z & &\in H^0(X_k, \cO(h - (l_{21} + l_{50}))), \\
 d_2 &= w + (-1-i)x + (i)y + (1-i)z & &\in H^0(X_k, \cO(h - (l_{50} + l_{88}))), \\
 b_1 &= w + (-i)x + (1+i)y + (1-i)z & &\in H^0(X_k, \cO(h - (l_{21} + l_{65}))), \\
 a   &= w + ix + (1+i)y + (-1+i)z & &\in H^0(X_k, \cO(h - (l_{63} + l_{65}))), 
\end{alignat*}
and
\[ \phi_2 = c_2 d_1 + (1+i) c_1 d_2 + c_2 d_2 \in H^0(X_k, \cO(2h - (l_{22} + l_{50} + l_{63} + l_{88}))),
\]
so that 
they vanish on the indicated lines.
Let $P = a c_1 d_2$, $Q = a c_2 d_1$, and $R = b_1 \phi_2$.

Then $P,Q,R$ form a basis of $H^0(X_k, \cO(D_1))$, 
and $\pi_1$ is given by $[P : Q : R]$.
We compute the images of the above curves and obtain
\begin{align*}
l_{10}, l_{18} \to S_{10,18} &= (0 :  0   : 1), \\
l_{16}, l_{99} \to S_{16,99} &= (1 :  0   : 1+i), \\
l_{29}, l_{49} \to S_{29,49} &= (1 :  1-i : 1-i), \\
l_{60}, l_{73} \to S_{60,73} &= (1 : -1-i : 0), \\
l_{23}         \to T_{ 23}   &= (0 :  1   : -1), \\
l_{37}         \to T_{ 37}   &= (1 : -1+i : 0), \\
l_{62}         \to T_{ 62}   &= (1 :  1+i : 0), \\
l_{68}         \to T_{ 68}   &= (1 :  1+i : -i), \\
l_{102}        \to T_{102}   &= (1 : -1+i : i), \\
l_{112}        \to T_{112}   &= (0 :  1   :-1-i), \\
\end{align*}
for each component.
We look for sextic curve that have these $10$ points as singular points.
By a straightforward calculation (computer-aided, omitted)
we observe that there is only one such sextic curve
and its equation is 
\begin{gather*}
G = 
(-1) Q^2 R^4
+ (-1+i) Q^3 R^3
+ Q^4 R^2
+ Q^5 R
+ (i) Q^6
+ (-i) P Q R^4 \\
+ (-i) P Q^2 R^3 
+ (-1-i) P Q^4 R
+ (-1-i) P Q^5
+ P^2 R^4
+ (-1) P^2 Q R^3 \\
+ (i) P^2 Q^3 R
+ (-1) P^3 R^3 
+ (1+i) P^3 Q^2 R
+ (-1+i) P^3 Q^3
+ (-1) P^4 R^2 \\
+ P^5 R
+ (1+i) P^5 Q
+ P^6
.
\end{gather*}
Hence $Y^2 = G(P,Q,R)$ is the equation of $X_k$ relative to $\pi_1$, at least after extending $k$.
By a calculation (omitted) we observe that the points $S_{j,j'}$ (resp.\ $T_j$) are exactly the cusps (resp.\ nodes) of the sextic,
hence their fibers are exactly $l_j \cup l_{j'}$ (resp.\ $l_j$).
It remains to check there are no other singular points on this sextic.
First we see that such singular point is necessarily $\bF_9$ ($= k$)-rational since,
if not, the fibers give classes of $\NS(X_{\overline k})$ that are not $\Gal(\overline \bF_9/\bF_9)$-invariant,
which is absurd because $\NS(X_{\overline k})$ is generated by lines defined over $\bF_9$.
So we only need to check $\bF_9$-rational points on $X_k$, 
and as there are only $91$ $\bF_9$-rational points in $\bP^2$,
this can be done in a finite amount of calculation (omitted).
\end{proof}

\end{document}